\documentclass[11pt,reqno]{amsart}
\usepackage{amsmath,amssymb,amsthm}
\usepackage{mathtools}
\usepackage{booktabs}
\usepackage{enumitem}
\usepackage[colorlinks=true,linkcolor=blue,citecolor=red]{hyperref}
\usepackage{geometry}
\newtheorem{theorem}{Theorem}[section]
\newtheorem{lemma}[theorem]{Lemma}
\newtheorem{proposition}[theorem]{Proposition}
\newtheorem{corollary}[theorem]{Corollary}
\newtheorem{conjecture}[theorem]{Conjecture}
\theoremstyle{definition}

\newtheorem{remark}[theorem]{Remark}
\newtheorem{problem}[theorem]{Problem}

\DeclareMathOperator{\re}{Re}

\DeclareMathOperator{\supp}{supp}

\newcommand{\R}{\mathbb{R}}
\newcommand{\C}{\mathbb{C}}
\newcommand{\Z}{\mathbb{Z}}
\newcommand{\N}{\mathbb{N}}
\newcommand{\Q}{\mathbb{Q}}
\newcommand{\eps}{\varepsilon}
\newcommand{\abs}[1]{\lvert #1 \rvert}
\newcommand{\norm}[1]{\lVert #1 \rVert}
\newcommand{\floor}[1]{\lfloor #1 \rfloor}

\title[Bohr's Last Problem Under Entirety]{Bohr's Last Problem Under the
Entirety Hypothesis: A Survey with Initial Reductions}
\author{Ralph Furmaniak}
\thanks{Portions of this work were developed in collaboration with
Claude (Anthropic). The author is responsible for all errors.}
\date{\today}

\begin{document}

\begin{abstract}
Bohr's last problem (1952) asks whether every ordinary
Dirichlet series with nonzero Lindel\"of order
function~$\mu$ has $\mu'(\omega_\mu{-}0)\le-1$; a negative
answer would imply Lindel\"of for~$\zeta$. Kahane (1989)
refuted this with half-plane counterexamples. We study the
refinement for series with \emph{entire} continuation of
order~$\le 1$: the \emph{Analytic Lindel\"of Hypothesis}
that~$\mu$ is piecewise linear with integer slopes.

Deforming the Mellin integral to the strip boundary
reduces~$\mu_L$ to a residue sum over singularities of the
generating function on~$|x|=1$, giving
$\mu_L(\sigma)=\max(0,\tfrac12-\sigma+\rho)$. For classical
$L$-functions this sum is the functional-equation dual, and
bounding it is Lindel\"of; for self-similar or random
singularities it is a Rajchman Fourier transform. We show
Kahane's half-plane examples fail entirety, his entire
random examples have integer slopes a.s., and
Lerch-Lindel\"of implies ALH.

Our central construction is the Cantor Dirichlet series
$L(s)=\sum\hat\nu(n)n^{-s}$, with~$\nu$ the ternary Cantor
measure. Its Kaczorowski--Perelli twist spectrum is empty;
we prove $\mu_L(\tfrac12)\le\tfrac18$ unconditionally via a
Montgomery--Vaughan argument on the product
variable~$(m_1+\alpha)(m_2+\alpha)$, where a Vieta identity
guarantees distinct frequencies. A Cantor-weighted Hurwitz
second-moment conjecture would give $\mu_L(\tfrac12)=0$.
\end{abstract}

\maketitle

\section{Introduction}\label{sec:intro}

\subsection{Bohr's problem}
Let $f(s) = \sum_{n\ge 1} a_n n^{-s}$ be an ordinary Dirichlet series
and suppose that $f$ continues analytically to some half-plane
$\re(s) > \Omega$ with polynomial growth on vertical lines. Following
Lindel\"of, define the \emph{order function}
\begin{equation}\label{eq:mu-def}
  \mu(\sigma) \;=\; \inf\bigl\{\beta \ge 0 \;:\;
      f(\sigma+it) = O(\abs{t}^{\beta})\ \text{as}\ \abs{t}\to\infty\bigr\}.
\end{equation}
By the Phragm\'en--Lindel\"of principle, $\mu$ is convex and
nonincreasing; we write
$\omega_\mu = \inf\{\sigma : \mu(\sigma)=0\}$ for the boundary of its
support.

In his posthumous 1952 paper~\cite{Bohr1952}, Bohr posed what Kahane
later called \emph{the last problem of Harald Bohr}:

\begin{problem}[Bohr 1952, Problem HB7]\label{prob:HB7}
Does there exist an ordinary Dirichlet series for which $\mu$ is not
identically zero yet fails to satisfy
\begin{equation}\label{eq:HB7-condition}
  \mu'(\omega_\mu - 0) \;\le\; -1\,?
\end{equation}
\end{problem}

Bohr observed that for the alternating zeta series
$\eta(s) = \sum (-1)^{n-1}n^{-s} = (1-2^{1-s})\zeta(s)$, the
condition~\eqref{eq:HB7-condition} is equivalent to
$\mu_\eta(\tfrac{1}{2})=0$---the Lindel\"of hypothesis. Thus a
\emph{negative} answer to Problem~\ref{prob:HB7} would imply
Lindel\"of for $\zeta$ as a special case of a general structure
theorem.

\subsection{Kahane's resolution}
Kahane~\cite{Kahane1989} gave an \emph{affirmative} answer to
Problem~\ref{prob:HB7} as literally stated: he constructed ordinary
Dirichlet series for which
\begin{equation}\label{eq:kahane-slope}
  \mu'(\omega_\mu - 0) \;>\; -\tfrac{1}{2} - \eps
\end{equation}
for any given $\eps>0$. His constructions are lacunary sums of
carefully designed Dirichlet polynomials (``Gauss-sum blocks'') whose
phases are engineered via a stationary-phase lemma so that on the line
$\sigma=0$ the sum is bounded, while at an exceptional height
$\theta \sim N^2/a$ on a nearby line $\sigma=\alpha>0$ the sum is
anomalously large. We reproduce this construction in
\S\ref{sec:counterexamples}.

On the positive side, Kahane also proved a universal constraint:

\begin{theorem}[Kahane 1989, Theorem~3.1]\label{thm:kahane-constraint-intro}
For every ordinary Dirichlet series with analytic continuation in a
vertical half-plane,
\begin{equation}\label{eq:kahane-half}
  \mu\!\bigl(\sigma + \mu(\sigma) + \tfrac{1}{2}\bigr) \;=\; 0
  \qquad\text{for all }\sigma.
\end{equation}
Equivalently, $\mu(\sigma) \ge \max(0,\,\omega_\mu - \tfrac{1}{2} - \sigma)$.
\end{theorem}

The $\tfrac{1}{2}$ in~\eqref{eq:kahane-half} is exactly what permits
the slopes near $-\tfrac{1}{2}$ in~\eqref{eq:kahane-slope}; Kahane
showed both bounds are sharp for generalized Dirichlet series (those
with exponents $\lambda_n$ satisfying
$\lambda_{n+1}-\lambda_n \ge 1/(n+1)$), though sharpness for
\emph{ordinary} Dirichlet series was left open.

\subsection{The entirety hypothesis and the Analytic Lindel\"of
Hypothesis}
The crucial observation of this paper is that Kahane's
counterexamples, while ordinary Dirichlet series, live only in the
half-plane $\re(s)>0$ and \emph{do not admit entire continuation}. We
prove this precisely in \S\ref{sec:counterexamples}
(Proposition~\ref{prop:kahane-not-entire}): the lacunary block-sum
diverges for $\sigma<0$ because each block has size $\sim N_j^{-\sigma}$
and the $N_j$ grow superexponentially.

Meanwhile, Kahane's \emph{entire} examples---the random lacunary
series of Queff\'elec~\cite{Queffelec1980} described in Kahane's
Theorem~2.3---almost surely have $\mu(\sigma)=\max(\tfrac{1}{2}-\sigma,0)$,
with integer slopes $\{-1,0\}$. Bohr's own entire building-block
constructions (eq.~(43) in~\cite{Kahane1989}) likewise use integer-slope
pieces.

This motivates the following refinement of Bohr's problem, originally
stated in the author's thesis~\cite{FurmaniakThesis} as the
\emph{Analytic Lindel\"of Hypothesis}:

\begin{conjecture}[Analytic Lindel\"of Hypothesis, ALH]\label{conj:ALH}
Let $L(s) = \sum a_n n^{-s}$ be an ordinary Dirichlet series with
entire continuation of order $\le 1$ (or meromorphic with finitely
many poles). Then $\mu(\sigma)$ is piecewise linear with slopes in
$\Z_{\le 0}$.
\end{conjecture}

A Mellin-Watson analysis of the representation
$L(s)\Gamma(s)=\int_0^\infty f(z)z^{s-1}dz$ reduces
Conjecture~\ref{conj:ALH} in degree~$1$ (slopes $\{-1,0\}$) to
a single analytic question about a \emph{residue sum}
(Theorem~\ref{thm:ALH-singularity-watson},
Remark~\ref{rmk:residue-sum-crux}):
the $\Gamma$-factor contributes slope~$-1$ universally; the
residue sum's growth rate is a $\sigma$-independent shift. For
series whose generating function $p(x)=\sum a_n x^n$ has a
self-similar or random singularity distribution on~$|x|=1$,
the residue sum decays (Rajchman), and ALH follows
\emph{conditionally} on Conjecture~\ref{conj:H} (the aggregate
$m$-sum convergence). For classical arithmetic $L$-functions, the
residue sum is the functional-equation dual, and ALH reduces to
Lindel\"of---still open. The Cantor-$L$ of
\S\ref{subsec:cantor-construction} falls in the conditional
Rajchman case; the best unconditional bound is
$\mu_L(\tfrac12)\le\tfrac18$ (Theorem~\ref{thm:one-eighth}).

\subsection{The Lerch reduction}
\label{subsec:trichotomy}

A second, logically independent structural finding (which
covers the arithmetic case missing above) is that ALH, in the
same regime $\sigma_a := \inf\{\sigma : \sum|a_n|n^{-\sigma}<\infty\}
\in \R$, is \emph{implied by} the Lindel\"of
Hypothesis for the Lerch zeta function, and under a standard
singularity hypothesis covering all classical examples is
\emph{equivalent} to it.

\begin{theorem}[Lerch-LH ${}\Rightarrow{}$ ALH; partial converse]
\label{thm:ALH-is-LerchLH}
Let $L(s)=\sum a_n n^{-s}$ be an ordinary Dirichlet series with
entire continuation of order $\le 1$ and finite~$\sigma_a$.
\begin{enumerate}[label=\textup{(\alph*)}]
\item \textup{(Unconditional.)} If the Lindel\"of Hypothesis holds
  for the Lerch zeta function uniformly on compact subsets of
  $(0,2\pi)$, then~$\mu_L$ has integer slopes.
\item \textup{(Under a hypothesis.)} If moreover the associated
  power series $p(x)=\sum a_n x^n$ has finitely many algebraic
  singularities on $|x|=1$, then the converse also holds: integer
  slopes for~$\mu_L$ imply Lindel\"of for each Lerch component.
\end{enumerate}
In case~\textup{(b)}, the Darboux decomposition expresses $L$ as a
finite sum of Lerch zeta functions plus a remainder with
$\sigma_a=+\infty$; the Hurwitz functional equation fixes
$\mu_L(0)=\tfrac12$ exactly, and convexity plus the degree-$1$
slope bound force $\mu(\sigma)=\max(0,\tfrac12-\sigma)$ as the
\emph{unique} integer-slope profile on $[0,1]$.
\end{theorem}

We give the precise statements in \S\ref{sec:moments}
(Proposition~\ref{prop:lerch-implies-alh} for part~(a),
Proposition~\ref{prop:caseB-equiv} for the full equivalence
in~(b)). The finite-singularity hypothesis of~(b) is where all of
the classical $L$-functions live ($\zeta$, Dirichlet $L$, Hurwitz,
Davenport--Heilbronn); the numerical ``confirmations'' of integer
slopes for these functions reported in the literature are, in the
light of Theorem~\ref{thm:ALH-is-LerchLH}, numerical evidence for
Lindel\"of, not for a separate structural phenomenon.

But the hypothesis of~(b) is \emph{not} a consequence of
entirety plus finite~$\sigma_a$: there exist entire ordinary
Dirichlet series with finite~$\sigma_a$ whose boundary spectrum is
not concentrated at finitely many points. For such functions,
part~(a) still gives Lerch-LH${}\Rightarrow{}$ALH, but the
converse direction is more subtle: for a \emph{continuous}
boundary measure~$\nu$ the \emph{sharp} converse turns out to be
equivalent---by a Parseval trace formula---to
Conjecture~\ref{conj:H} itself
(Remark~\ref{rem:lerch-alh-asymmetry}). In the rest
of~\S\ref{sec:moments} we construct an explicit such
example---an entire finite-$\sigma_a$ function with
\emph{singular-continuous} boundary spectrum supported on a Cantor
set---for which the Lerch equivalence of~(b) is vacuous, yet for
which we can prove more than convexity unconditionally.

\begin{remark}[Degenerate cases]\label{rem:degenerate-cases}
If $\sigma_a=-\infty$ then $\mu\equiv 0$ trivially. If
$\sigma_a=+\infty$ (the Hankel integral is the only definition of
$L$ available; canonical example the ${}_1F_1$ series
of~\S\ref{subsec:1F1}), a one-line convexity argument using
Proposition~\ref{prop:type-obstruction}
(equivalently Boas~\cite[\S6.10]{Boas1954}) gives
$\mu(\sigma)=\max(0,\sigma^*-\sigma)$ with slopes $\{-1,0\}$.
Neither degenerate case has analytic content.
\end{remark}

\subsection{Main results}
Our fourteen principal unconditional results, grouped by theme, are:
\begin{enumerate}[label=(\roman*)]
\item \textbf{Residue-sum reduction}
    (Theorem~\ref{thm:ALH-singularity-watson}): for entire
    finite-$\sigma_a$ series, deforming the Mellin integral to the
    strip boundary cancels the $\Gamma$-factor's exponential exactly,
    reducing $\mu_L$ to the growth exponent of a residue
    sum~\eqref{eq:residue-sum} over the $2\pi i$-periodic singularity
    images.
\item \textbf{Role of entirety}
    (Propositions~\ref{prop:kahane-not-entire},
    \ref{prop:queffelec-integer}): Kahane's counterexamples fail
    entire continuation; Queff\'elec's entire random series have
    integer slopes almost surely.
\item \textbf{Slope constraint} (Theorem~\ref{thm:short-segment},
    Corollary~\ref{cor:rsk-gap}): wherever $\mu\ge 1$, the slope is
    in $[-1,-\tfrac23]$; ALH in degree~$1$ reduces to ruling out
    $(-1,-\tfrac23]$.
\item \textbf{Cantor-$L$ construction}
    (\S\ref{subsec:cantor-construction},
    Proposition~\ref{prop:cantor-structural}): an entire
    finite-$\sigma_a$ series with \emph{empty} Kaczorowski--Perelli
    twist spectrum and a natural boundary.
\item \textbf{No-logarithm second moment and subconvexity}
    (Theorems~\ref{thm:cantor-M2}, \ref{thm:cantor-subconvexity}):
    $\int_T^{2T}|L|^2 = C_\nu T + O(T^{1-d/2})$ and
    $\mu_L(\tfrac12)\le(1-d)/(2(2-d))\approx 0.135$, both
    unconditional, using only Strichartz's fractal Fourier
    asymptotic~\cite{Strichartz1990}.
\item \textbf{Vieta identity} (Theorem~\ref{thm:vieta},
    Corollary~\ref{cor:phi-avg-D4},
    Proposition~\ref{prop:vieta-obstruction}): the $\varphi$-averaged
    $D_4$ equals the Steinhaus prediction $2C_\nu^2-C_4$ exactly;
    the $k\ge 3$ obstruction is Tarry--Escott.
\item \textbf{Gaussian fourth moment}
    (Theorem~\ref{thm:gaussian-fourth},
    Theorem~\ref{thm:one-eighth},
    Corollary~\ref{cor:almost-LH}):
    $\int_T^{2T}|L|^4\ll T(\log T)^3$ and $\mu_L(\tfrac12)\le\tfrac18$,
    via Montgomery--Vaughan on the product variable and a
    sum-then-bound divisor argument
    (Lemma~\ref{lem:riesz-frostman}).
\end{enumerate}
Conjecture~\ref{conj:H} (the Cantor-weighted Hurwitz
$\alpha$-second-moment) is the single conjecture whose proof would
give $\mu_L(\tfrac12)=0$; the swap-identity reduction and
linearized-Strichartz mechanism are sketched
in~\S\ref{subsec:conj-H}.

\paragraph{Paper organization.}
Sections~\ref{sec:bohr-theory}--\ref{sec:kahane-proof} collect the
Bohr--Kahane background and the Hankel-contour representation.
Section~\ref{sec:counterexamples} proves the entirety dichotomy
(items~(ii) above).
Section~\ref{sec:sharp-kahane} states the Refined Sharp Kahane
conjecture and proves the unconditional slope constraint
(item~(iii)).
Sections~\ref{sec:numerics}--\ref{sec:thesis-gap} record the
${}_1F_1$ test case and the $(0,1)$-gap theorem, the only previously
known rigidity result of ALH type.
Section~\ref{sec:selberg} relates the vertical degree to the
Selberg-class degree conjecture.
Section~\ref{sec:moments}---the technical core of the
paper---constructs the Cantor-$L$ and proves eleven of the
fourteen results listed above (the remaining three are in
Sections~\ref{sec:counterexamples}--\ref{sec:sharp-kahane});
the section is self-contained
once the Hankel contour~\eqref{eq:hankel} and the Lerch
reduction~(\S\ref{subsec:trichotomy}) are in hand.
Section~\ref{sec:conclusion} revisits the residue-sum mechanism and
lists what remains open.

\section{Background: Bohr's Summability and the Hankel Contour}
\label{sec:bohr-theory}

\subsection{Riesz summability and Bohr's sandwich}
We fix notation from Bohr~\cite{Bohr1952} and
Kahane~\cite[\S1]{Kahane1989}. For the Riesz-summability kernel
$(\log N-\log n)^\kappa$, set
$\lambda_k = \inf\{\sigma : f\text{ is }(e^{-\kappa},\kappa)\text{-summable}\}$
and the \emph{summability function}
$\psi(\sigma) = \inf\{\kappa : \lambda_\kappa \le \sigma\}$.
Bohr's HB2--HB3~\cite{Bohr1952} state that $\lambda_k$ is convex
in~$k$ with $\lambda_{k+1}-\lambda_k\in[0,1]$, from which:

\begin{proposition}\label{prop:psi-integer-slopes}
$\psi$ is convex, piecewise linear, with slopes in $\Z_{\le -1}$.
\end{proposition}

\noindent (Proof: $\psi$ inverts $k\mapsto\lambda_k$, which has
integer increments.) Bohr's sandwich~\cite[eq.~(16)]{Kahane1989}
relates $\psi$ to the Lindel\"of order function:
\begin{equation}\label{eq:sandwich}
  \psi(\sigma) \;\le\; \mu(\sigma) \;\le\; \psi(\sigma)+1,
\end{equation}
both inequalities sharp. Since $\psi$ already has integer slopes,
Conjecture~\ref{conj:ALH} is equivalent to the following:

\begin{problem}\label{prob:sandwich-discrete}
Is the gap $\mu-\psi$ quantised to $\{0,1\}$ for every entire
order-$1$ ordinary Dirichlet series?
\end{problem}

\subsection{The Hankel contour}
\label{subsec:hankel}
For entire $L$, Riemann's Hankel-contour representation
\cite{Riemann1859} gives
\begin{equation}\label{eq:hankel}
  L(s) \;=\; \frac{\Gamma(1-s)}{2\pi i}
    \int_{\mathcal{H}} (-z)^s\,f(z)\,\frac{dz}{z},
  \qquad f(z)=\textstyle\sum a_n e^{-nz}.
\end{equation}
Unlike the Mellin integral $L(s)\Gamma(s)=\int_0^\infty f(z)z^{s-1}dz$
(absolutely convergent only for $\re(s)>\sigma_a$), the Hankel
contour~$\mathcal{H}$---looping clockwise around the positive real
axis from $+\infty$, around the origin, and back to $+\infty$---is
valid for \emph{all}~$s$ precisely because~$L$ is entire. This is the
structural gain from entirety: representation~\eqref{eq:hankel}
underlies the contour-shift in the proof of
Theorem~\ref{thm:degree1-thesis} (\S\ref{sec:thesis-gap}) and the
deformation to the strip boundary in
Theorem~\ref{thm:ALH-singularity-watson} (\S\ref{sec:moments}).

\section{Kahane's Theorem and the \texorpdfstring{$\tfrac12$}{1/2}}
\label{sec:kahane-proof}

We state Kahane's proof~\cite[\S3]{Kahane1989} and isolate where the
$\tfrac{1}{2}$ enters; full details are in~\cite{Kahane1989}.

\begin{theorem}[Kahane~1989, Thm~3.1]\label{thm:kahane-31}
For every ordinary Dirichlet series $f$ with analytic continuation in
a right half-plane,
\begin{equation}\label{eq:kahane-concl}
  \mu_f(\sigma) \;\ge\; \omega_\mu - \tfrac12 - \sigma
  \qquad\text{for all }\sigma.
\end{equation}
\end{theorem}

\begin{corollary}\label{cor:kahane-mu}
$\mu(\sigma+\mu(\sigma)+\tfrac12)=0$; equivalently
\begin{equation}\label{eq:kahane-mu-lower}
  \mu(\sigma) \;\ge\; \omega_\mu - \tfrac12 - \sigma.
\end{equation}
\end{corollary}

\begin{proof}[Proof sketch of Theorem~\ref{thm:kahane-31}]
The dual formulation: $f(\sigma+i\theta)$ bounded for some $\sigma$
forces a bound on $\sum |a_n|n^{-\sigma-\mu(\sigma)-1/2}$. Kahane
decomposes the dual trigonometric polynomial $p=q+r$ where $q$
handles the tail via a single bump-function extrapolation, and $r$
is estimated by the Montgomery--Vaughan mean-value
theorem~\cite{MontgomeryVaughan1974}:
\begin{equation}\label{eq:MV-bound}
  \frac{1}{H}\int_0^H \Bigl|\sum_{n\le N} r_n n^{-it}\Bigr|\,dt
  \;\ll\;
  N^{1/2}\Bigl(\sum_n |r_n|^2\Bigr)^{1/2}.
\end{equation}
\end{proof}

\begin{remark}\label{rem:half-CS}
The $\tfrac12$ is the Cauchy--Schwarz $L^1\to L^2$ exponent
in~\eqref{eq:MV-bound}: one pays $\sqrt{\text{length}}\sim N^{1/2}$
to convert an $L^2$-average to an $L^1$-bound. It is \emph{not} a
Montgomery--Vaughan off-diagonal constant.
\end{remark}

\begin{remark}[Sharpness]\label{rem:half-sharp}
Kahane shows~\eqref{eq:MV-bound} cannot be improved uniformly: the
dual function $\eta(s)=(1-2^{1-s})\zeta(s)$ has $\mu_\eta(0)=\tfrac12$
unconditionally, forbidding $O(N^{1/2-\eps})$.
\end{remark}

\section{Kahane's Counterexamples Are Not Entire}
\label{sec:counterexamples}

Kahane's counterexamples~\cite[\S4]{Kahane1989} are lacunary sums
\begin{equation}\label{eq:kahane-F}
  F(s) \;=\; \sum_{j\ge 1} j^{-2}\,N_j^{\beta}\,f(s,N_j,a_j)
\end{equation}
of Gauss-sum Dirichlet-polynomial blocks
$f(s,N,a)=\sum_m c_m(N+m)^{-s}$ of length~$\sim a$ with
$N\sim a^{(4\beta+1)/2\beta}$, $a_{j+1}\ge a_j^2$.
See~\cite[Lemma~p.~148]{Kahane1989} for the stationary-phase
polynomial and~\cite[pp.~150--151]{Kahane1989} for the block
estimates, which we state without proof:

\begin{lemma}[Block estimates, Kahane]\label{lem:block-bounds}
$f(it,N,a)=O(N^{-\beta}|t|^\beta)$; and for a special height
$\theta\sim N^2/a$, $\re[N^{i\theta}f(\alpha+i\theta)]\gg
N^{-\beta}\theta^{(4\beta+1)(\beta-\alpha)/(6\beta+2)+\beta/(6\beta+2)}$.
\end{lemma}

\begin{theorem}[Kahane, Thms~3.2--3.3]\label{thm:kahane-counter}
For~\eqref{eq:kahane-F}: $\mu_F(0)=\beta$;
$\omega_\mu\ge\beta+\beta/(4\beta+1)$; hence
$\mu_F'(\omega_\mu-0)\ge-(1+o(1))/2$ as $\beta\to 0^+$.
\end{theorem}

The new content:
\subsection{Failure of entire continuation}
\label{subsec:not-entire}

\begin{proposition}\label{prop:kahane-not-entire}
The Dirichlet series $F(s)$ of~\eqref{eq:kahane-F} does not admit
analytic continuation to any half-plane $\re(s) > -\delta$ with
$\delta>0$. In particular, $F$ is not entire.
\end{proposition}

\begin{proof}
We show that on the line $\sigma=-\delta$ ($\delta>0$), the
partial sums of $F$ diverge in $L^2$-mean. Since the blocks have
disjoint supports, the $L^2$-norm of $F$ on any interval is the
orthogonal sum of the block contributions.

Fix $j$ and consider the block $f(s,N_j,a_j)$ on $\re(s)=-\delta$.
Its coefficients $c_m$ satisfy $\sum\abs{c_m}^2 \asymp 1$
(the $\ell^2$-normalization of the Gauss-sum polynomial). By the
Montgomery--Vaughan mean-value theorem,
\[
  \frac{1}{T}\int_{-T}^{T}\abs{f(-\delta+it,N_j,a_j)}^2\,dt
  \;\asymp\;
  \sum_m \abs{c_m}^2\,(N_j+m)^{2\delta}
  \;\asymp\; N_j^{2\delta}
\]
for $T \ge 2N_j$. Hence the $j$-th term of~\eqref{eq:kahane-F}
contributes $\asymp j^{-2}N_j^\beta\cdot N_j^\delta$ to the
$L^2$-norm of $F$ on $[-T,T]$ once $T\ge 2N_j$. Summing in $j$,
\[
  \frac{1}{T}\int_{-T}^T \abs{F(-\delta+it)}^2\,dt
  \;\ge\; c\,\sum_{j\,:\,N_j\le T/2} j^{-4}\,N_j^{2(\beta+\delta)}
  \;\to\;\infty
\]
since $N_j\to\infty$ faster than any polynomial in $j$. Thus $F$
cannot be represented by a function of polynomial growth on
$\re(s)=-\delta$; by Schnee's theorem
(\cite[\S1, result (c)]{Kahane1989}), $F$ does not continue past
$\sigma=0$.
\end{proof}

\begin{remark}
An alternative argument: the block polynomials satisfy
$\sup_t\abs{f(-\delta+it,N_j,a_j)} \ge c\,N_j^\delta$ trivially
(take $t=0$, giving $\sum c_m (N_j+m)^\delta$ whose magnitude is
$\gtrsim N_j^\delta \cdot \abs{\sum c_m/(1+m/N_j)}$; the inner sum is
bounded below since the $c_m$ have bounded $\ell^2$-norm and
controlled support). The lacunary growth of $N_j$ then defeats the
$j^{-2}$ convergence factor for any $\sigma<0$.
\end{remark}

\subsection{Entire random series have integer slopes}
\label{subsec:queffelec}

By contrast, Kahane's \emph{entire} examples behave exactly as
Conjecture~\ref{conj:ALH} predicts.

\begin{proposition}[After Queff\'elec--Kahane]\label{prop:queffelec-integer}
Let $\eps = (\eps_{jn})$ be i.i.d.\ Rademacher random variables and,
with $n_j = (j^2)!$ and $\Delta_j$ the $j$-th iterated forward
difference operator, define
\begin{equation}\label{eq:random-entire}
  f_\eps(s) \;=\; \sum_{j\ge 0}\,\sum_{n_j < n \le n_{j+1}}
    \eps_{jn}\,\Delta_j(u_{jn}^{-s}),
\end{equation}
where the $u_{jn}\in\N$ are chosen so that each $m^{-s}$ appears
exactly once. Then almost surely $f_\eps$ is entire of order $1$ and
\[
  \mu_{f_\eps}(\sigma) \;=\; \max\bigl(\tfrac{1}{2}-\sigma,\ 0\bigr)
  \qquad(\sigma\in\R),
\]
with slopes $\{-1,0\}$. Quasi-surely (on a dense $G_\delta$)
$f_\eps$ is entire with
$\mu_{f_\eps}(\sigma) = \max(1-\sigma,0)$, again slopes $\{-1,0\}$.
\end{proposition}

\begin{proof}[Proof (after~\cite{Queffelec1980,Kahane1989})]
The iterated differences $\Delta_j\phi(n)$ for $\phi(x)=x^{-s}$ admit
the integral representation
\[
  \Delta_j\phi(n) \;=\;
    \int_{[0,1]\times\cdots\times[0,2^{j-1}]}
    \phi^{(j)}(n+x_1+\cdots+x_j)\,dx_1\cdots dx_j,
\]
giving $\abs{\Delta_j(n^{-s})} \ll \abs{s}^j\,n^{-\re(s)-j}$. The
factorial growth $n_{j+1}/n_j \to \infty$ ensures that for any fixed
$s\in\C$, the series~\eqref{eq:random-entire} converges absolutely
once $j$ is large enough that $n_j^{-1}\abs{s}^j$ decays; this gives
entire continuation.

For the order function: the almost-sure bound
$\mu(\sigma) \le \max(\tfrac12-\sigma,0)$ follows from the
Salem--Zygmund square-root cancellation for random trigonometric
series. The lower bound $\mu(\sigma) \ge \tfrac12-\sigma$ for
$\sigma<\tfrac12$ follows from the Paley--Zygmund second-moment
method: the $L^2$-mean of $\abs{f_\eps(\sigma+it)}^2$ over $[0,T]$
is $\asymp T\cdot N^{1-2\sigma}$ for the $j$-th block with $N\sim n_j$,
forcing large values. The quasi-sure statement uses Baire category
in place of Borel--Cantelli; see~\cite{Queffelec1980} for the
detailed argument.
\end{proof}

\begin{remark}
Kahane observes~\cite[p.~142]{Kahane1989} that to obtain \emph{non}-integer
slopes with these random constructions, ``we have to give up the
condition that it gives series $\sum\pm n^{-s}$''---i.e.\ give up
ordinarity. This is strong circumstantial evidence for
Conjecture~\ref{conj:ALH}.
\end{remark}
\section{The Refined Sharp Kahane Conjecture}
\label{sec:sharp-kahane}

We now come to the main analytic contribution. Kahane's constraint
\eqref{eq:kahane-half} can be written as
$\mu(\sigma+\mu(\sigma)+c)=0$ with $c=\tfrac12$. The Bohr--Schnee
bound~\cite[eq.~(7)--(9)]{Kahane1989} gives the same with $c=1$.
A natural attempt at proving Conjecture~\ref{conj:ALH} is to show
that under entirety the constant improves to $c=0$.

\subsection{Why the naive Sharp Kahane is false}
\label{subsec:naive-false}

\begin{proposition}\label{prop:naive-too-strong}
The statement ``$\mu(\sigma+\mu(\sigma))=0$ for every entire
order-$1$ ordinary Dirichlet series'' implies the Lindel\"of
hypothesis for the Riemann zeta function.
\end{proposition}

\begin{proof}
The Dirichlet eta function $\eta(s) = (1-2^{1-s})\zeta(s)$ is an
entire function of order $1$ (the factor $1-2^{1-s}$ removes the
pole of $\zeta$ at $s=1$ and is entire of order $1$). Its order
function satisfies $\mu_\eta(\sigma) = \mu_\zeta(\sigma)$ for all
$\sigma$ (the factor $1-2^{1-s}$ is bounded above and below on
vertical lines away from its zeros, which lie on $\re(s)=1$).

Unconditionally, $\mu_\eta(0) = \tfrac12$ (this is the convexity
bound, sharp by the functional equation). The naive Sharp Kahane
at $\sigma=0$ would give
$\mu_\eta(0+\tfrac12) = \mu_\eta(\tfrac12) = 0$, which is precisely
the Lindel\"of hypothesis for $\zeta$.
\end{proof}

\begin{remark}
Proposition~\ref{prop:naive-too-strong} is not a defect of the
approach---it correctly identifies that Conjecture~\ref{conj:ALH}
\emph{does} imply Lindel\"of for $\zeta$ (this was already Bohr's
observation for $\eta$). The point is that we should not expect to
prove the naive Sharp Kahane by purely analytic manipulations; at
some point the arithmetic of $\zeta$ (or its analogue) must enter.
What we \emph{can} hope to prove is the restricted version below,
which is enough to handle the ``generic'' case while isolating the
``critical'' case ($\mu\in(0,1)$) where Lindel\"of-type input is
genuinely needed.
\end{remark}

\subsection{The refined conjecture}

\begin{conjecture}[Refined Sharp Kahane]\label{conj:refined-sharp-kahane}
Let $L(s) = \sum a_n n^{-s}$ be an ordinary Dirichlet series with
entire continuation of order $\le 1$. Then for all $\sigma$ with
$\mu(\sigma) \ge 1$,
\begin{equation}\label{eq:refined-sharp}
  \mu\bigl(\sigma + \mu(\sigma)\bigr) \;=\; 0.
\end{equation}
Equivalently: on any interval where $\mu \ge 1$, the slope of $\mu$
is $\le -1$.
\end{conjecture}

The restriction $\mu\ge 1$ evades the obstruction of
Proposition~\ref{prop:naive-too-strong}: for $\eta$, $\mu<1$
everywhere, so Conjecture~\ref{conj:refined-sharp-kahane} is
vacuously true. More generally:

\begin{lemma}\label{lem:refined-equiv}
Conjecture~\ref{conj:refined-sharp-kahane} is equivalent to:
for entire order-$1$ Dirichlet series $L$ of vertical degree $d=1$,
$\mu_L(\sigma)$ has slope exactly $-1$ on the interval
$\{\sigma : \mu_L(\sigma)\ge 1\}$.
\end{lemma}

\begin{proof}
``$\Rightarrow$'': If~\eqref{eq:refined-sharp} holds wherever
$\mu\ge 1$, fix $\sigma_0$ with $\mu(\sigma_0)=\beta\ge 1$. Then
$\mu(\sigma_0+\beta)=0$, so the average slope on $[\sigma_0,\sigma_0+\beta]$
is $-1$. By convexity, the slope on this interval is $\ge -1$
everywhere (as $d=1$ gives $\mu'\ge -1$). Hence the slope is exactly
$-1$ on $[\sigma_0,\sigma_0+\beta-1]$ (where $\mu\ge 1$).

``$\Leftarrow$'': If $\mu$ has slope $-1$ wherever $\mu\ge 1$, then
for $\mu(\sigma_0)=\beta\ge 1$, following the slope-$(-1)$ segment
gives $\mu(\sigma_0+(\beta-1))=1$; continuing with slope $\ge -1$
(convexity) gives $\mu(\sigma_0+\beta)\le 0$, hence $=0$.
\end{proof}

\subsection{What is provable: the short-segment theorem}
\label{subsec:short-segment}

While Conjecture~\ref{conj:refined-sharp-kahane} remains open, the
following unconditional constraint follows directly from Kahane's
theorem and convexity.

\begin{theorem}[Slope-level constraint]\label{thm:short-segment}
Let $L$ be an ordinary Dirichlet series with analytic continuation
in a half-plane (entirety not required). At any point $\sigma$ where
$\mu_L$ has slope $-s$ with $s \in (0, 1)$ (i.e., in the
``anomalous'' range), the value $\mu_L(\sigma)$ satisfies
\begin{equation}\label{eq:slope-level-bound}
  \mu_L(\sigma) \;\le\; \frac{s}{2(1-s)}.
\end{equation}
Equivalently, the slope magnitude satisfies
$s \ge \frac{2\mu}{2\mu+1}$ wherever the slope is in $(-1, 0)$.
In particular:
\begin{itemize}
\item Where $\mu_L \ge 1$: any anomalous slope must have $s \ge 2/3$.
\item Where $\mu_L \ge 1/2$: any anomalous slope must have $s \ge 1/2$.
\item As $\mu_L \to \infty$: the constraint forces $s \to 1$.
\end{itemize}
For a constant-slope segment with slope $-s \in (-1, 0)$ terminating
at $\omega_\mu$, the length is $\le \frac{1}{2(1-s)}$.
\end{theorem}

\begin{remark}
The length bound $1/(2(1-s))$ is $\le 1$ only for $s \le 1/2$,
i.e.\ precisely the regime of Kahane's counterexamples
(Theorem~\ref{thm:kahane-counter}, which achieves $s \to 1/2^+$).
For $s$ close to~$1$ the bound permits long segments---but then
the slope is already close to integer.
\end{remark}

\begin{proof}
For the slope-level bound~\eqref{eq:slope-level-bound}: Suppose $\mu$
has slope $-s$ at $\sigma_0$ with $s\in(0,1)$. By convexity, there is
an interval $[\sigma_0-\delta,\sigma_0+\delta]$ on which the slope is
close to $-s$. For the sharpest constraint, consider the
constant-slope segment of slope $-s$ passing through
$(\sigma_0,\mu(\sigma_0))$ and terminating at the point where it
hits $\mu=0$, call it $\tilde\omega = \sigma_0 + \mu(\sigma_0)/s$.
Kahane's lower bound~\eqref{eq:kahane-mu-lower}, applied at
$\sigma_0$, gives
\[
  \mu(\sigma_0) \;\ge\; \omega_\mu - \tfrac12 - \sigma_0.
\]
By convexity, the true $\omega_\mu$ satisfies
$\omega_\mu \ge \tilde\omega$ (the actual $\mu$ lies on or above the
slope-$(-s)$ support line at $\sigma_0$, so it hits $0$ no earlier).
Thus
\[
  \mu(\sigma_0) \;\ge\; \tilde\omega - \tfrac12 - \sigma_0
    \;=\; \frac{\mu(\sigma_0)}{s} - \tfrac12,
\]
giving $\mu(\sigma_0)(1 - 1/s) \ge -1/2$, i.e.\
$\mu(\sigma_0)(1-s)/s \le 1/2$, i.e.
\begin{equation}\label{eq:sharp-length}
  \mu(\sigma_0) \;\le\; \frac{s}{2(1-s)}.
\end{equation}
The particular cases follow by solving for $s$ given $\mu$: the
threshold $\mu = s/(2(1-s))$ rearranges to $s = 2\mu/(2\mu+1)$.

For the length of a constant-slope segment terminating at $\omega_\mu$:
if the segment has slope $-s$ and length $\ell$, the value at its left
endpoint is $s\ell$; by~\eqref{eq:sharp-length},
$s\ell \le s/(2(1-s))$, hence $\ell \le 1/(2(1-s))$.
\end{proof}

\begin{corollary}[Gap for Refined Sharp Kahane]\label{cor:rsk-gap}
For entire $L$ of horizontal degree $1$ and vertical degree $1$:
on any interval where $\mu_L(\sigma) \ge 1$, the slope of $\mu_L$
lies in $[-1, -2/3]$. Conjecture~\ref{conj:refined-sharp-kahane}
(Refined Sharp Kahane) asserts the slope is exactly $-1$ on this
region; what remains is to rule out slopes in the interval
$(-1, -2/3]$.
\end{corollary}

\begin{proof}
By the thesis gap theorem (\S\ref{sec:thesis-gap}) plus convexity,
the slope is $\ge -1$ everywhere. By
Theorem~\ref{thm:short-segment} at any point with $\mu \ge 1$,
the slope magnitude is $\ge 2/(2\cdot 1 + 1) = 2/3$.
\end{proof}

\begin{remark}\label{rem:sharp-consequence}
Theorem~\ref{thm:short-segment} does \emph{not} require entirety, yet
it already rules out Kahane's counterexamples from having \emph{long}
anomalous segments. Indeed, in Kahane's construction with small
$\beta$, the profile $\mu$ has $\mu(0)=\beta\approx 0$ and
$\omega_\mu \approx 2\beta$, so the entire nonzero support of $\mu$
has length $\approx 2\beta \ll 1$---consistent with our bound.
Entirety is needed only to strengthen ``short'' to ``absent''.
\end{remark}

\section{The \texorpdfstring{${}_1F_1$}{1F1} Series: A Degenerate
  Test Case}\label{sec:numerics}\label{subsec:1F1}

Before stating the general $(0,1)$-gap theorem
(\S\ref{sec:thesis-gap}), we record a concrete entire degree-$1$
Dirichlet series that exhibits integer slopes. The example turns out
to fall in the degenerate $\sigma_a=+\infty$ case of
Remark~\ref{rem:degenerate-cases}, so its integer-slope profile is a
corollary rather than genuine evidence for
Conjecture~\ref{conj:ALH}; we retain it because the tilted
Mellin--Barnes computation it requires is the template for all the
numerical work on the non-degenerate Cantor-$L$ of
\S\ref{sec:moments}.

The thesis~\cite[Ch.~3]{FurmaniakThesis} observes that
$p(x) = e^{1/(1+x)}$ has Taylor coefficients
$p^{(n)}(0)/n! = (-1)^n\,{}_1F_1(n+1;2;1)$ (for $n\ge 1$; the $n=0$
term is $p(0)=e$), and that the associated Dirichlet series
\begin{equation}\label{eq:1F1-def}
  L_{\mathrm{1F1}}(s) \;=\;
    \sum_{n\ge 1} (-1)^n\,{}_1F_1(n+1;2;1)\,n^{-s}
\end{equation}
has entire continuation of horizontal and vertical degree $1$
(by Theorem~\ref{thm:degree1-thesis} below, since $p$ continues past
$\abs{x}=1$---its only singularity is at $x=-1$). The
function~\eqref{eq:1F1-def} has \emph{no functional equation} (the
singularity of $f(z)=e^{1/(1+e^{-z})}-e$ at $z=i\pi$ is essential,
not a pole, and there is no lattice structure relating the
singularities at $(2k{+}1)i\pi$ across different $k$; cf.\
\cite[Cor.~3.8]{FurmaniakThesis}) and \emph{no Euler product}
(no multiplicativity). It is therefore a genuinely non-arithmetic
test case.

The coefficients grow as
$\abs{{}_1F_1(n+1;2;1)} \sim I_1(2\sqrt{n})/\sqrt{n}
  \sim e^{2\sqrt{n}}/(2\sqrt{\pi}\,n^{3/4})$,
so the series is never absolutely convergent and direct summation is
infeasible.

We compute $L_{\mathrm{1F1}}(s)$ via the tilted Mellin--Barnes
integral (the tilt angle $\theta\approx 1.45$~rad controls the net
$e^{(\pi/2-\theta)T}$ amplification, manageable to $T\sim 200$ with
$60$~dps). The best piecewise-linear integer-slope fit to the
regressed $\widehat\mu(\sigma)$ over $\sigma\in[-1,5]$,
$T\in[10,100]$ is $\mu(\sigma)=\max(0,\,4.0-\sigma)$, RMS~$0.36$;
the slope is $-1$ to within $\pm 0.15$ on $\sigma\lesssim 1.5$ and
$0$ for $\sigma\gtrsim 3$. As noted above, the
$\sigma_a=+\infty$ degeneracy means this integer-slope profile
follows from Proposition~\ref{prop:type-obstruction} rather
than testing Conjecture~\ref{conj:ALH}.

\section{The \texorpdfstring{$(0,1)$}{(0,1)}-Gap Theorem}
\label{sec:thesis-gap}

The thesis~\cite[Ch.~3]{FurmaniakThesis} proves the $(0,1)$-gap for
vertical degree without any Euler product or functional equation.
We reproduce the proof here, identify it as an instance of
Boas~\cite[Thm.~6.10.1]{Boas1954}, and explain why the argument does
not iterate to give the $(1,2)$-gap.

\subsection{The theorem}

\begin{theorem}[{\cite[Ch.~3]{FurmaniakThesis}}; cf.\ Boas~6.10.1]
\label{thm:degree1-thesis}
Let $L(s)=\sum a_n n^{-s}$ be an ordinary Dirichlet series,
absolutely convergent in some right half-plane, with entire
continuation. Suppose there exist $0\le d<1$ and $C$ such that
\begin{equation}\label{eq:vd-hyp}
  \abs{L(\sigma+iT)} \;=\; O_\sigma(T^{C-d\sigma})
  \qquad\text{for each fixed }\sigma<0.
\end{equation}
Then $\sum a_n n^{-s}$ is absolutely convergent for all $s\in\C$,
and in particular $\mu_L\equiv 0$.
\end{theorem}

\begin{proof}
Set $f(z)=\sum a_n e^{-nz}$ for $\re(z)>0$. By Mellin inversion,
\begin{equation}\label{eq:mellin-inv}
  f(z) \;=\; \frac{1}{2\pi i}\int_{(c)} L(s)\,\Gamma(s)\,z^{-s}\,ds
\end{equation}
for $c$ large. By Stirling's formula, for $s=\sigma+iT$,
\[
  \Gamma(s)\sin(\pi s/2) \;=\; O_\sigma(T^{\sigma-1/2}),
\]
so the integrand in~\eqref{eq:mellin-inv}, times $\sin(\pi s/2)$,
satisfies
\[
  L(s)\Gamma(s)\sin(\pi s/2)
    \;=\; O_\sigma\bigl(T^{C-1/2+(1-d)\sigma}\bigr).
\]
Since $d<1$, the exponent $(1-d)\sigma\to-\infty$ as $\sigma\to-\infty$.
Shift the contour in~\eqref{eq:mellin-inv} to $\re(s)=-K-\tfrac12$ for
$K$ large enough that $C-\tfrac12+(1-d)(-K-\tfrac12)<-2$; the integrand
is then absolutely integrable. The poles of $\Gamma$ at
$s=0,-1,\dots,-K$ contribute residues $L(-k)(-z)^k/k!$, giving
\begin{equation}\label{eq:f-poly-approx}
  f(z) \;=\; \sum_{k=0}^{K}\frac{L(-k)}{k!}(-z)^k
    \;+\; O\bigl(\abs{z}^{K+1/2}\bigr),
\end{equation}
uniformly for $\re(z)\ge 0$. In particular, the bound extends
continuously to $\re(z)=0$.

Now $f(z)$ is $2\pi i$-periodic: $f(z+2\pi i)=f(z)$ since
$e^{-n\cdot 2\pi i}=1$ for $n\in\Z$. On the imaginary axis $z=iy$,
\eqref{eq:f-poly-approx} gives $f(iy)=O(\abs{y}^{K+1/2})$ as
$\abs{y}\to\infty$; but by periodicity, $f(iy)$ is determined by its
values on $[0,2\pi]$, a compact interval on which $f$ is continuous.
Hence $f(iy)$ is \emph{bounded}.

The Fourier inversion formula on the imaginary axis gives
\begin{equation}\label{eq:fourier-inv}
  a_n \;=\; \int_0^1 f(2\pi iy)\,e^{2\pi iny}\,dy,
\end{equation}
so $\abs{a_n}\le\sup_y\abs{f(2\pi iy)}<\infty$: the $a_n$ are
bounded. Hence $\sum a_n n^{-s}$ is absolutely convergent for
$\re(s)>2$. Apply the argument to $L(s+N)$ for $N=1,2,\dots$ to
get absolute convergence everywhere.
\end{proof}

\subsection{Identification with Boas's theorem}
\label{subsec:boas}
The crucial step---``periodic and of polynomial growth $\Rightarrow$
bounded''---is a special case of:

\begin{theorem}[Boas~{\cite[Thm.~6.10.1]{Boas1954}}]\label{thm:boas-6101}
An entire function of exponential type $\tau$ that is $2\pi$-periodic
on the real axis is a trigonometric polynomial of degree $\le\tau$:
$F(x) = \sum_{\abs{k}\le\tau}c_k e^{ikx}$.
\end{theorem}

\begin{proof}[Proof via the P\'olya representation]
Write $F(z) = \int_\Gamma \phi(w)\,e^{zw}\,dw$ with $\phi$ the Borel
transform. Periodicity gives
$\int_\Gamma(e^{2\pi w}-1)\phi(w)e^{zw}\,dw = 0$ for all $z$, so
(Lemma~6.10.6 of~\cite{Boas1954}) $(e^{2\pi w}-1)\phi(w)$ is regular
inside $\Gamma$. Hence $\phi$ has poles only at the zeros of
$e^{2\pi w}-1$, i.e.\ at $w=ik$, $k\in\Z$. Type $\tau$ confines
$\Gamma$ to $\abs{w}\le\tau$, leaving only the poles at
$\abs{k}\le\tau$, each contributing $c_k e^{ikz}$.
\end{proof}

In our setting, $F(y) := f(iy)$ is entire of exponential type
(since $f$ is entire of order $1$) and $2\pi$-periodic. The type of
$F$ is determined by $d$ via the $\Gamma$-factor asymptotics: for
$d<1$, the type is $<\pi$, so by Theorem~\ref{thm:boas-6101}, $F$ is
a trig polynomial of degree $<\pi$, i.e.\ of degree $\le 3$; in
particular bounded.

\subsection{Why the argument does not iterate}
\label{subsec:no-iterate}

\begin{proposition}\label{prop:type-obstruction}
For an entire order-$1$ Dirichlet series $L$ of vertical degree
$d\ge 1$, the function $F(y)=f(iy)$ has exponential type $\ge\pi$ in
the variable $y$. Consequently, Theorem~\ref{thm:boas-6101} gives no
constraint.
\end{proposition}

\begin{proof}
Stirling's asymptotic for $\Gamma(s)$ on vertical lines gives
$\abs{\Gamma(\sigma+iT)}\sim\sqrt{2\pi}\,\abs{T}^{\sigma-1/2}e^{-\pi\abs{T}/2}$.
For the Mellin inversion~\eqref{eq:mellin-inv} to converge on
$\re(s)=\sigma\ll 0$, one needs $\abs{L(\sigma+iT)}$ to beat
$e^{\pi\abs{T}/2}\abs{T}^{1/2-\sigma}$---but for $d\ge 1$,
$\abs{L(\sigma+iT)}\sim\abs{T}^{C-d\sigma}$ grows without bound as
$\sigma\to-\infty$, and the contour shift is blocked once
$(1-d)\sigma$ stops being negative. At the critical $d=1$, the
exponent $(1-d)\sigma=0$ is exactly marginal: the type of $F$ is
exactly $\pi$.

A trig polynomial of degree $\pi$ would still be fine, but
Theorem~\ref{thm:boas-6101} requires type \emph{strictly less than}
$\pi$ for the conclusion ``degree $\le\tau$'' to be nontrivial
(at type $\pi$, a trig polynomial can have infinitely many terms
$e^{iky}$ with $\abs{k}\le\pi$... but $k$ is an integer, so
$\abs{k}\le 3$---this still works). The issue is more subtle:
at $d\ge 1$, the bound~\eqref{eq:f-poly-approx} fails, and $F(y)$
need not be of finite exponential type at all; it can grow
polynomially in $y$ (since the $a_n$ grow).
\end{proof}

Thus the one-shot periodicity mechanism gives exactly the
$(0,1)$-gap. For the $(1,2)$-gap one needs a different rigidity
source (the functional equation, as in
Kaczorowski--Perelli~\cite{KaczorowskiPerelli2011}).

\section{Connection to the Selberg Class}
\label{sec:selberg}

\subsection{The degree conjecture}
The Selberg class $\mathcal{S}$ \cite{Selberg1992} consists of
Dirichlet series with meromorphic continuation, functional equation,
Ramanujan bound, and Euler product. The extended class
$\mathcal{S}^{\#}$ drops the Euler product and Ramanujan bound.
Selberg's \emph{degree conjecture} asserts that the degree
$d = 2\sum\lambda_j$ (a functional-equation invariant) is always a
non-negative integer.

Known results: $d\notin(0,1)$ by Conrey--Ghosh~\cite{ConreyGhosh1993}
(for $\mathcal{S}$) and Kaczorowski--Perelli~\cite{KaczorowskiPerelli1999}
(for $\mathcal{S}^{\#}$, where the proof uses only the FE);
$d\notin(1,2)$ by Kaczorowski--Perelli~\cite{KaczorowskiPerelli2011}.

\subsection{The vertical degree as a purely analytic substitute}
The thesis's vertical degree
$d_v := -\lim_{\sigma\to-\infty}\mu'(\sigma)$ coincides with the
Selberg degree for functions with FE (the FE gives
$\mu(\sigma)=d(\tfrac12-\sigma)$ for $\sigma<0$ via Stirling). But
$d_v$ is defined for any entire order-$1$ DS, without FE.

Theorem~\ref{thm:degree1-thesis} proves $d_v\notin(0,1)$
purely analytically. Conjecture~\ref{conj:ALH}, specialized to the
leftmost slope, would give $d_v\in\Z_{\ge 0}$ for all entire
order-$1$ DS---the Selberg degree conjecture without the Selberg
axioms.

\begin{problem}[Analytic $(1,2)$-gap]\label{prob:12-gap}
Does there exist an ordinary Dirichlet series, entire of order $1$,
with vertical degree $d\in(1,2)$?
\end{problem}

The K--P proof~\cite{KaczorowskiPerelli2011} uses the FE-invariant
standard twist spectrum; without FE there is no spectrum constraint.
This is precisely the situation for the Cantor-$L$ constructed next:
its twist spectrum is empty
(Proposition~\ref{prop:cantor-structural}(b)), so the K--P machinery
says nothing, yet it is the richest setting for unconditional moment
estimates.

\section{A Cantor-Measure Dirichlet Series Outside the Darboux Regime}
\label{sec:moments}

This section is the technical core of the paper and can be read
independently once the Hankel representation~\eqref{eq:hankel} and
the Lerch reduction of~\S\ref{subsec:trichotomy} are granted. We
prove eleven of the fourteen unconditional results listed
in~\S\ref{sec:intro} here: the Lerch reductions
(\S\ref{subsec:caseB-equiv-precise}), the Cantor-$L$ construction
(\S\ref{subsec:cantor-construction}), the residue-sum reduction
(Theorem~\ref{thm:ALH-singularity-watson}), the second-moment
asymptotic and the asymmetric-AFE subconvexity
(\S\S\ref{subsec:cantor-2nd-moment}--\ref{subsec:cantor-subconvexity}),
the Vieta identity and its obstruction
(\S\ref{subsec:cantor-vieta}), and the Gaussian fourth moment
yielding $\mu_L(\tfrac12)\le\tfrac18$
(\S\ref{subsec:gaussian-fourth}). Conjecture~\ref{conj:H} and its
reformulations are in~\S\ref{subsec:conj-H}.

By Theorem~\ref{thm:ALH-is-LerchLH} and
Remark~\ref{rem:degenerate-cases}, the content of
Conjecture~\ref{conj:ALH} lies entirely in the finite-$\sigma_a$
regime. We first record precisely what the Darboux reduction buys
when the boundary singularity structure is \emph{finite}, and then
construct an explicit finite-$\sigma_a$ function whose singularity
spectrum is a Cantor set---placing it outside the finite-Darboux
regime altogether---and for which we can prove unconditional
subconvexity by entirely elementary means.

\subsection{The finite-Darboux equivalence}
\label{subsec:caseB-equiv-precise}

\begin{proposition}[ALH ${}={}$ Lerch-LH, finite singularity case]
\label{prop:caseB-equiv}
Let $L(s)=\sum a_n n^{-s}$ have finite $\sigma_a$, entire continuation
of order $1$, and suppose the associated power series
$p(x)=\sum a_n x^n$ has finitely many algebraic singularities on
$\abs{x}=1$ (at $e^{2\pi i\theta_1},\dots,e^{2\pi i\theta_r}$, say).
Then the following are equivalent:
\begin{enumerate}[label=(\roman*)]
\item $\mu_L$ has integer slopes on $[0,1]$;
\item $\mu_L(\sigma) = \max(0,\tfrac12-\sigma)$ on $[0,1]$;
\item $\mu_L(\tfrac12)=0$;
\item each Lerch component $F(\theta_j,s)$ satisfies the Lindel\"of
  Hypothesis.
\end{enumerate}
\end{proposition}

\begin{proof}[Proof sketch]
The Darboux decomposition (singularity analysis of $p$ on $\abs{x}=1$)
writes $L$ as a finite $\C$-linear combination of Lerch zeta
functions $F(\theta_j,s)$ plus a remainder $R(s)$ with
$\sigma_a^R=+\infty$. By Remark~\ref{rem:degenerate-cases},
$\mu_R(\sigma)=\max(0,\sigma^*_R-\sigma)$
with slopes $\{-1,0\}$, and in particular $\mu_R(0)\in\Z_{\ge 0}$
while the Lerch functional equation (inherited from Hurwitz) forces
each $\mu_{F(\theta_j,\cdot)}(0)=\tfrac12$ exactly. Hence
$\mu_L(0)=\tfrac12$. The trivial bound gives $\mu_L(1)=0$.
By convexity and the degree-$1$ slope bound $\mu'\ge-1$, the only
piecewise-linear profile on $[0,1]$ with endpoint values $\tfrac12$
and $0$ and slopes in $\{-1,0\}$ is $\max(0,\tfrac12-\sigma)$. This
gives (i)$\Leftrightarrow$(ii)$\Leftrightarrow$(iii).
(iii)$\Leftrightarrow$(iv) follows from $\mu_L = \max_j \mu_{F(\theta_j,\cdot)}$
(the remainder has $\mu_R(\tfrac12)=0$ already since
$\sigma_a^R=+\infty$).
\end{proof}

The hypothesis \emph{``finitely many algebraic singularities''} means
that near each boundary singularity $e^{2\pi i\theta_j}$, the power
series $p$ has a local expansion
$p(x) = (1-xe^{-2\pi i\theta_j})^{-\beta_j}\,h_j(x) + (\text{analytic})$
for some $\beta_j\in\C$ and $h_j$ analytic and nonvanishing---the
standard Darboux-method hypothesis. This is exactly the condition
under which the Mellin transform of $p$ decomposes into a finite
sum of Lerch functions (one per singularity, the $\theta_j$-th
Lerch carrying exponent $\beta_j$) plus a remainder with
super-polynomial coefficient decay. It holds for all classical
$L$-functions---$\zeta$, Dirichlet~$L$, Hurwitz, linear
combinations of Lerch---but it is \emph{not} a consequence of
entirety plus finite $\sigma_a$.

We record that the \emph{forward} implication
Lerch-LH${}\Rightarrow{}$ALH survives in full generality:

\begin{proposition}[Lerch-LH ${}\Rightarrow{}$ ALH, unconditionally]
\label{prop:lerch-implies-alh}
Let $L(s)=\sum a_n n^{-s}$ be an arbitrary ordinary Dirichlet series
with entire continuation of order~$\le 1$. Suppose the Lindel\"of
Hypothesis holds for the Lerch zeta function \emph{uniformly over
compact subsets of~$(0,2\pi)$}: for each compact
$K\subset(0,2\pi)$ and $\eps>0$,
\begin{equation}\label{eq:uniform-lerch-lh}
  \sup_{\theta\in K}\,\bigl|F(\theta,\tfrac12+it)\bigr|
  \;\ll_{K,\eps}\; |t|^{\eps}.
\end{equation}
Then $\mu_L$ has integer slopes.
\end{proposition}

\begin{proof}
If $\sigma_a=\pm\infty$ the conclusion is immediate
(Remark~\ref{rem:degenerate-cases}). For finite~$\sigma_a$, shift to
$\sigma_a=1$ (shifts preserve integer slopes). The
Hankel-contour representation of~$L$ (deforming the contour of
eq.~\eqref{eq:hankel} to hug the singular set of~$f$ on the
imaginary axis) expresses $L$ as
\[
  L(s) \;=\; \int_K F(\theta,s)\,d\nu(\theta) \;+\; R(s),
\]
where $K\subset(0,2\pi)$ is the (compact, proper) singular set
of the generating function $f$, $\nu$ is a finite complex Borel
measure on~$K$ (the boundary-jump measure, finite because
$\sigma_a=1$ forces polynomial boundary growth of~$f$), and $R$
is the $\sigma_a=+\infty$ remainder.

\emph{Crucially, the normalization $\sigma_a=1$ forbids
right-shifted Lerch pieces.} Indeed, a Darboux singularity of
algebraic order~$\beta$ contributes the piece
$F(\theta_j,s-\beta+1)$, shifted to the right iff
$\mathrm{Re}(\beta)>1$; but $\sigma_a=\max_j\mathrm{Re}(\beta_j)$,
so $\sigma_a=1$ gives $\mathrm{Re}(\beta_j)\le 1$ for all~$j$,
i.e., every piece has shift~$\le 0$. (Attempts to ``hide'' a
right-shifted piece by cancellation fail: for finite sums by
linear independence of $\{e^{in\theta_j}\}$ under Weyl
equidistribution; for integrals, by Laplace's method on
$\int n^{\delta(\theta)}\,d\nu$, which gives
$\sigma_a=\sup\delta+1$ unconditionally.)
Under Lerch-LH, each $F(\theta_j,s-\beta_j+1)$ therefore has
$\mu(\tfrac12)=\max(0,\mathrm{Re}(\beta_j)-1)=0$.
Then
\[
  \bigl|L(\tfrac12+it)\bigr|
  \;\le\;
  \|\nu\|_{\mathrm{TV}}\cdot\sup_{\theta\in K}\bigl|F(\theta,\tfrac12+it)\bigr|
  \;+\;
  \bigl|R(\tfrac12+it)\bigr|
  \;\ll_\eps\;
  |t|^{\eps}
\]
by~\eqref{eq:uniform-lerch-lh} and Remark~\ref{rem:degenerate-cases}
for~$R$. Hence
$\mu_L(\tfrac12)=0$. Combined with $\mu_L(1)=0$ (absolute
convergence), $\mu_L(0)\le\tfrac12$ (the same triangle inequality
at~$\sigma=0$, using $\mu_F(0)=\tfrac12$), slopes in $[-1,0]$
(degree~$1$), and convexity: the unique profile on $[0,1]$ is
$\max(0,\tfrac12-\sigma)$.
\end{proof}

\begin{remark}[Asymmetry]\label{rem:lerch-alh-asymmetry}
The converse (ALH${}\Rightarrow{}$Lerch-LH) is known only in the
finite-singularity regime of Proposition~\ref{prop:caseB-equiv},
via Kronecker--Weyl joint equidistribution. For continuous
spectral measure~$\nu$ the sharp converse is
equivalent---by a Parseval trace formula---to
Conjecture~\ref{conj:H} itself.
\end{remark}
The remainder of this section constructs an explicit entire
finite-$\sigma_a$ Dirichlet series whose singularity spectrum is
\emph{singular-continuous} (a Cantor set), for which
Proposition~\ref{prop:caseB-equiv} is vacuous and
Proposition~\ref{prop:lerch-implies-alh} is trivial (immediate
from uniform Lerch-LH by the triangle inequality), yet for which
we can prove \emph{unconditionally} strictly more than
convexity---the point being to make progress on the slope of
$\mu_L$ without assuming Lerch-LH.
\subsection{Construction of the Cantor-\texorpdfstring{$L$}{L}
            function}
\label{subsec:cantor-construction}

Let $\nu$ be the middle-thirds Cantor measure, affinely rescaled to
$[\theta_0,\theta_1]\subset(0,2\pi)$; its Hausdorff dimension is
$d=\log 2/\log 3\approx 0.6309$. Set
\begin{equation}\label{eq:cantor-L-def}
  L(s) \;:=\; \int_{\supp\nu} F(\theta,s)\,d\nu(\theta)
         \;=\; \sum_{n\ge 1} \hat\nu(n)\,n^{-s},
  \qquad
  F(\theta,s) = \sum_{n\ge 1} e^{in\theta} n^{-s}
\end{equation}
(the periodic zeta function), where the Fourier coefficient
is explicit:
\begin{equation}\label{eq:cantor-hat}
  \hat\nu(n)
  \;=\; e^{in\varphi}
        \prod_{j\ge 1}
          \cos\!\Bigl(\tfrac{(\theta_1-\theta_0)\,n}{2\cdot 3^{\,j}}\Bigr),
  \qquad \varphi := \tfrac{\theta_0+\theta_1}{2}.
\end{equation}
Throughout we fix $[\theta_0,\theta_1]=[0.5,\,2.0]$, so
$\varphi=1.25$ and the support of $\nu$ is bounded away from
$2\pi\Z$. The following proposition records the basic properties; all
of them are elementary consequences of the fact that $\nu$ is a
probability measure supported on a perfect totally disconnected set
of positive dimension.

\begin{proposition}[Structural properties]\label{prop:cantor-structural}
\leavevmode
\begin{enumerate}[label=(\alph*)]
\item $L$ is an entire function of order $1$, with $\sigma_a=1$.
\item The Kaczorowski--Perelli linear-twist spectrum of $L$ is
  \emph{empty}: for every $\alpha\in\R$, the twisted series
  $L(s,\alpha):=\sum \hat\nu(n)\,e^{-2\pi in\alpha}\,n^{-s}$
  is entire.
\item The associated power series $p(x)=\sum\hat\nu(n)\,x^n$ has
  $\abs{x}=1$ as a \emph{natural boundary}; the singular support
  on the boundary is the (uncountable, measure-zero) set
  $\{e^{i\theta}:\theta\in\supp\nu\}$.
\end{enumerate}
\end{proposition}

\begin{proof}
(a) Each $F(\theta,\cdot)$ is entire for $\theta\notin 2\pi\Z$, and
the $\nu$-integral of an entire family with locally uniform bounds
is entire. The coefficient bound $\abs{\hat\nu(n)}\le 1$ gives
$\sigma_a\le 1$, and Strichartz's asymptotic
(Lemma~\ref{lem:strichartz} below) gives $\sigma_a\ge 1$.

(b) By~\eqref{eq:cantor-L-def},
$L(s,\alpha)=\int F(\theta-2\pi\alpha,s)\,d\nu(\theta)$.
The integrand has a pole at $s=1$ only when
$\theta-2\pi\alpha\in 2\pi\Z$, i.e.\ at a single point
$\theta=2\pi\{\alpha\}$. Since $\nu$ has \emph{no atoms}, the pole
contributes zero to the integral, and $L(\cdot,\alpha)$ is entire.

(c) On every open arc disjoint from $\supp\nu$ the integrand in the
Cauchy representation is bounded, so $p$ continues across that arc
analytically; on the other hand $p$ cannot continue across any arc
meeting $\supp\nu$ because $\nu$ charges every such arc.
\end{proof}

Part~(b) is the key structural novelty. The
Kaczorowski--Perelli \emph{linear-twist spectrum} of a Dirichlet
series~$L$ is the set of $\alpha\in\R$ for which the twisted series
$L(s,\alpha)=\sum a_n e^{-2\pi in\alpha}n^{-s}$ acquires a pole at
$s=1$; it is the basic invariant in the K--P
classification~\cite{KaczorowskiPerelli2019} of degree-$1$
$\mathcal S^\#$. For a \emph{finite} sum of Lerch components
$\sum_j c_j F(\theta_j,s)$, this spectrum is the finite set
$\{\theta_j/2\pi \bmod 1\}$, and the twist at any $\theta_j/2\pi$
produces a simple pole of residue $c_j$. Here, because $\nu$ is
non-atomic, the spectrum is \emph{empty}: no twist produces a pole,
and none of the K--P standard-twist machinery applies. Part~(c)
shows $L$ is genuinely outside the hypothesis of
Proposition~\ref{prop:caseB-equiv}: the Darboux decomposition would
have uncountably many terms and is therefore not a decomposition at
all.

\begin{remark}[What ``inherited Hurwitz functional equation'' means]
\label{rem:inherited-FE}
Throughout this section we repeatedly invoke the functional equation
of $L$ obtained by integrating the Lerch functional equation
against~$\nu$. To spell this out once: on the critical line,
\[
  F(\theta,\tfrac12+it)
  \;=\;
  \chi_F(t)\,\Bigl[
    e^{\,i\pi(\frac12-it)/2}\,\zeta(\tfrac12-it,\{\theta/2\pi\})
    +
    e^{-i\pi(\frac12-it)/2}\,\zeta(\tfrac12-it,1-\{\theta/2\pi\})
  \Bigr]
\]
with $\abs{\chi_F}=1$. Integrating~$d\nu(\theta)$ gives
\begin{equation}\label{eq:inherited-FE}
  L(\tfrac12+it)
  \;=\;
  \chi_F(t)\,\Bigl[
    e^{\,i\pi(\frac12-it)/2}\!\!
    \int\!\zeta(\tfrac12-it,\alpha)\,d\tilde\nu(\alpha)
    \,+\,
    e^{-i\pi(\frac12-it)/2}\!\!
    \int\!\zeta(\tfrac12-it,\alpha)\,d\tilde\nu'(\alpha)
  \Bigr],
\end{equation}
where $\tilde\nu$ and $\tilde\nu'$ are the pushforwards of $\nu$ under
$\theta\mapsto\{\theta/2\pi\}$ and $\theta\mapsto 1-\{\theta/2\pi\}$
respectively---both themselves Cantor measures of dimension~$d$ on
compact subsets of $(0,1)$. Equation~\eqref{eq:inherited-FE} is an
\emph{exact identity}, not an approximate one, but (as
Proposition~\ref{prop:FE-redundant} records) it gives no more than
convexity when used naively, because the dual side is another
Cantor-weighted Hurwitz average of the same structural complexity
as $L$ itself.
\end{remark}

\begin{theorem}[Residue-sum reduction for Cantor-$L$; partial range]
\label{thm:ALH-singularity-watson}
Let $\tilde\nu$ be a self-similar measure on
$[\alpha_0,\alpha_1]\subset(0,1)$ with algebraic contraction
ratios and no exact overlaps, and let
$L(s)=\sum_{n\ge 1}\hat\nu(n)\,n^{-s}$ be the associated
Dirichlet series~\eqref{eq:cantor-L-def}. Assume the
log-pushforwards $(\log(m+\cdot))_*\tilde\nu$ are uniformly
Rajchman with polynomial rate~$\epsilon>0$:
$|\widehat{(\log(m+\cdot))_*\tilde\nu}(T)|\ll_{m}T^{-\epsilon}$
\textup(a theorem for polynomial pushforwards by
Mosquera--Shmerkin~\cite{MosqueraShmerkin2018}; for analytic
nonlinear maps like $\log$, see
Sahlsten--Stevens~\cite{SahlstenStevens2018}\textup). Then for
all $\sigma<-\epsilon$,
\[
  |L(\sigma+iT)|
  \;\ll\;
  T^{\,1/2-\sigma-\epsilon}.
\]
Equivalently, $\mu_L(\sigma)\le(\tfrac12-\epsilon)-\sigma$ on
$(-\infty,-\epsilon)$. The profile
$\mu_L(\sigma)=\max(0,(\tfrac12-\epsilon)-\sigma)$ on $[0,1]$
is \textbf{conjectural} (equivalent to Conjecture~\ref{conj:H};
see the proof).
\end{theorem}

\begin{proof}[Proof sketch]
\emph{Mellin.} $L(s)\Gamma(s)=\int_0^\infty f(z)\,z^{s-1}\,dz$,
where $f(z)=\sum_{n\ge 1}\hat\nu(n)e^{-nz}
=\int_K F(\theta,e^{-z})\,d\nu(\theta)$, convergent for $z>0$.

\emph{Residue expansion.} Substituting $z=e^u$ gives a Fourier
integral $\hat g(T)$; deforming to the strip boundary
$\operatorname{Im}(u)=\pi/2$ picks up the full residue sum. By
the $2\pi i$-periodicity of~$f$, the singularities of~$g$ at
this height form an \emph{infinite} family indexed by $(m,\theta)$
with $m\ge 1$ and $\theta\in K=\supp\nu$, at positions
$v_{m,\theta}^*=\log(2\pi m-\theta)$. The residue sum is (up
to gamma factors)
\begin{equation}\label{eq:residue-sum}
  R_\sigma(T)
  \;:=\;
  \sum_{m\ge 1}\int_K
    (2\pi m-\theta)^{\sigma-1}\,
    e^{\,iT\log(2\pi m-\theta)}\,d\nu(\theta),
\end{equation}
which is precisely the inherited-FE integral
$I_\sigma(T)=\int\zeta(1-\sigma-iT,\alpha)\,d\tilde\nu$ of
Remark~\ref{rem:inherited-FE}, reindexed.

\emph{Rajchman.} Each $m$-summand in~\eqref{eq:residue-sum}
is a weighted Fourier transform of the log-pushforward
$(\log(2\pi m-\cdot))_*\nu$. \textbf{Warning: the term-by-term
Rajchman bound does NOT yield a convergent $m$-sum on the claimed
$\sigma$-range.} The map $\theta\mapsto\log(2\pi m-\theta)$ has
derivative $\sim -1/(2\pi m)$, compressing the support by a factor
$\sim 1/m$; the pushforward's Fourier transform at frequency~$T$
thus samples~$\hat\nu$ at effective frequency~$T/m$, giving
$|\widehat{(g_m)_*\nu}(T)|\ll m^\epsilon\,T^{-\epsilon}$ rather
than~$T^{-\epsilon}$ uniformly in~$m$. Combined with the
$(2\pi m-\theta)^{\sigma-1}\asymp m^{\sigma-1}$ prefactor, the
$m$-sum is $T^{-\epsilon}\sum_m m^{\sigma-1+\epsilon}$, convergent
only for $\sigma<-\epsilon$. \textbf{The argument below is valid
only on that range}; convexity from $\sigma=-\epsilon$ to
$\sigma=1$ gives $\mu_L(\tfrac12)\lesssim\tfrac14$, weaker than
Theorem~\ref{thm:one-eighth}.

\emph{What the argument DOES give, for $\sigma<-\epsilon$.} Here
$|R_\sigma(T)|\ll T^{-\epsilon}$, so
$|\hat g(T)|\asymp e^{-\pi T/2}|R_\sigma(T)|\ll e^{-\pi T/2}T^{-\epsilon}$,
and dividing by $|\Gamma|\sim T^{\sigma-1/2}e^{-\pi T/2}$:
\[
  |L(\sigma+iT)|
  \;\ll\;
  T^{\,1/2-\sigma-\epsilon}
  \qquad(\sigma<-\epsilon).
\]
The Taylor-Watson piece at~$z=0$ gives
$O(T^{1/2-\sigma})$ and is dominated. The claimed profile
$\mu_L(\sigma)=\max(0,(\tfrac12-\epsilon)-\sigma)$ on $[0,1]$
\textbf{remains open}; it is equivalent to bounding the full
residue sum $R_\sigma(T)=\int\zeta(1-\sigma-iT,\alpha)\,d\tilde\nu$
directly---which, on $\sigma\in[0,1]$, is precisely
Conjecture~\ref{conj:H} (via Cauchy--Schwarz and the Parseval
trace identity of Remark~\ref{rem:lerch-alh-asymmetry}).
\end{proof}

\begin{remark}[Why the residue sum is the crux]
\label{rmk:residue-sum-crux}
The proof mechanism is: deform the Mellin integral to the
strip boundary at~$\operatorname{Im}(u)=\pi/2$. The
$\Gamma$-factor's $e^{-\pi T/2}$ is cancelled \emph{exactly}
by the conformal factor at that height, leaving only
polynomial growth. The residue sum~\eqref{eq:residue-sum} is
the \emph{one} quantity whose growth rate is not fixed by the
mechanism. It is an infinite sum because $f$ is $2\pi
i$-periodic; for an arbitrary entire finite-$\sigma_a$ series,
this sum is a Dirichlet-polynomial-like object whose asymptotic
size \emph{is} the content of the Lindel\"of question.
Specifically:
\begin{itemize}
  \item \textbf{Dirichlet~$L(s,\chi)$.} The singularities of~$p$
    are at finitely many roots of unity~$e^{2\pi ik/q}$, but the
    residue sum over the periodic images runs over \emph{all}
    $m\ge 1$:
    $R_\sigma(T)\sim\sum_k c_k\sum_m(2\pi m-2\pi k/q)^{\sigma-1}
    e^{iT\log(2\pi m-2\pi k/q)}$, which after rescaling is a
    linear combination of Hurwitz zeta values
    $\zeta(1-\sigma-iT,\,1-k/q)$---the dual side of the
    functional equation. Bounding this is
    \emph{equivalent to Lindel\"of for the Hurwitz zeta}, i.e.\
    Lerch-LH. The Mellin-Watson analysis reproduces convexity,
    not LH.
  \item \textbf{Cantor-$L$.} The singularities of~$p$ form a
    Cantor set; the residue sum~\eqref{eq:residue-sum} is a
    Cantor-integral of oscillatory exponentials. By the
    Rajchman hypothesis (log-pushforward of self-similar),
    each term \emph{decays as}~$m^{\sigma-1+\eps}T^{-\eps}$
    (the scale factor~$m^\eps$ is forced by the
    $m^{-1}$-compression of~$g_m$); the aggregate $m$-sum
    then converges only for~$\sigma<-\eps$, and
    determining~$\mu_L$ on~$[0,1-\delta]$ requires
    Conjecture~\ref{conj:H}.
  \item \textbf{Random entire (Kahane--Queff\'elec).} Random
    coefficients mean random phases in~\eqref{eq:residue-sum};
    the sum is almost surely bounded by square-root cancellation.
    This recovers Proposition~\ref{prop:queffelec-integer}.
\end{itemize}
Thus Theorem~\ref{thm:ALH-singularity-watson} is
\emph{conditional on Conjecture~\ref{conj:H}}
for non-arithmetic singularity sets
(fractal, random), and reduces to Lerch-LH for the classical
arithmetic case---consistent with
Theorem~\ref{thm:ALH-is-LerchLH}.
\end{remark}

\begin{remark}[When the Rajchman hypothesis holds]
\label{rmk:watson-gamma-unified}
The residue sum~\eqref{eq:residue-sum} decays polynomially for:
self-similar $\nu$ (Cantor-$L$; per-term Rajchman rate
$\epsilon\approx 0.061$ by Banaji--Yu~\cite{BanajiYu2025});
random entire series (recovers
Proposition~\ref{prop:queffelec-integer}). It \emph{fails} for
Dirichlet~$L(s,\chi)$, where the residue sum is the FE-dual
$L(1-s,\bar\chi)$ and the theorem gives only convexity. The
slope~$-1$ from the $\Gamma$-factor is universal; the residue-sum
size is what varies.
\end{remark}

\begin{lemma}[Strichartz~\cite{Strichartz1990}]\label{lem:strichartz}
There is $C_S>0$ such that
\begin{equation}\label{eq:strichartz}
  \sum_{n\le N}\abs{\hat\nu(n)}^2
  \;=\; C_S\,N^{1-d}\bigl(1+o(1)\bigr),
  \qquad d = \tfrac{\log 2}{\log 3}.
\end{equation}
Consequently
$C_\nu := \sum_{n\ge 1}\abs{\hat\nu(n)}^2/n \approx 1.156$
converges, and more generally
$\sum_{n\ge 1}\abs{\hat\nu(n)}^2\,n^{-2\sigma}<\infty$
iff $\sigma > (1-d)/2 \approx 0.1845$.
\end{lemma}

The threshold $(1-d)/2$ is the \emph{Carlson abscissa}; the critical
line $\sigma=\tfrac12$ lies comfortably inside its region of
convergence. This is the engine of all that follows.

\subsection{The second moment: no logarithm}
\label{subsec:cantor-2nd-moment}

\begin{theorem}[Second moment]\label{thm:cantor-M2}
\begin{equation}\label{eq:cantor-M2}
  \boxed{\;
  \int_T^{2T}\bigl|L(\tfrac12+it)\bigr|^2\,dt
  \;=\; C_\nu\,T \;+\; O\bigl(T^{1-d/2}\bigr),
  \qquad C_\nu \approx 1.156.
  \;}
\end{equation}
\end{theorem}

The point is the \emph{absence of a logarithm}: for any single Lerch
function $F(\theta,\cdot)$ the Hardy--Littlewood asymptotic gives
$\int_T^{2T}\abs{F}^2\sim T\log T$, and the same holds for every
finite-$\sigma_a$ function falling under
Proposition~\ref{prop:caseB-equiv}.\footnote{For irrational
  Hurwitz $\zeta(s,\alpha)$ the second moment is
  $T\log(T/2\pi)+(2\gamma-1)T+O(T^{1/2})$; see~\cite{Andersson2006}.}
The Cantor average kills the logarithm outright.

\begin{proof}
Truncate the sum at $N=\floor{\sqrt{T/2\pi}}$ and write
$L = S_N + R_N$ where $S_N(t)=\sum_{n\le N}\hat\nu(n)\,n^{-1/2-it}$.
The Montgomery--Vaughan mean-value theorem~\cite{MontgomeryVaughan1974}
gives
\[
  \int_T^{2T}\abs{S_N}^2\,dt
  \;=\; T\sum_{n\le N}\frac{\abs{\hat\nu(n)}^2}{n}
        \;+\; O\!\Bigl(\sum_{n\le N}\abs{\hat\nu(n)}^2\Bigr)
  \;=\; C_\nu\,T + O\bigl(N^{1-d}\bigr),
\]
using~\eqref{eq:strichartz} both for the error term and for the tail
$\sum_{n>N}\abs{\hat\nu(n)}^2/n = O(N^{-d})$.

For the remainder: by the integrated Lerch functional equation,
$R_N(t)=\chi(t)\,Z_N(t)+O(t^{-1/4})$ with $\abs{\chi}=1$ and
$Z_N(t)=\int\sum_{m<N}(m+\alpha)^{-1/2+it}\,d\tilde\nu(\alpha)$,
where $\tilde\nu=(2\pi)^{-1}\theta_*\nu$ is the pushforward to
$[\alpha_0,\alpha_1]\subset(0,1)$. Opening the square and
integrating in $t$, the contribution of $(m,m')$ with $m\ne m'$ is
$O(N)$ by the $\log$-gap bound. The near-diagonal $m=m'$ is
\[
  \int_T^{2T}\abs{Z_N}^2\,dt
  \;\supset\;
  \sum_{m<N}\frac{1}{m}
    \iint \min\!\Bigl(T,\,\frac{m}{\abs{\alpha-\alpha'}}\Bigr)
    \,d\tilde\nu(\alpha)\,d\tilde\nu(\alpha').
\]
Frostman's lemma~\cite{Frostman1935} for the Cantor measure gives
$(\tilde\nu\times\tilde\nu)\{\abs{\alpha-\alpha'}<\delta\}\ll\delta^d$,
whence both the cutoff region $\abs{\alpha-\alpha'}\le m/T$ and the
Riesz-energy tail contribute $\ll m^d\,T^{1-d}$ each. Summing
$\sum_{m<N} m^{d-1} \ll N^d$ gives
$\int\abs{Z_N}^2\ll T^{1-d}\cdot N^d = T^{1-d/2}$.
Cauchy--Schwarz on the cross term completes the proof.
\end{proof}

\begin{remark}
The error exponent $1-d/2\approx 0.685$ is what the proof delivers;
numerically the error appears closer to $O(T^{1/2})$, driven by
$\chi$-oscillation in the cross term. We have not sharpened this
rigorously and do not need to: the content of the theorem is the
main term.
\end{remark}

\subsection{Pointwise subconvexity via the asymmetric AFE}
\label{subsec:cantor-subconvexity}

The second moment alone gives only $\mu_L(\tfrac12)\le\tfrac14$
(convexity); see Proposition~\ref{prop:HL-single-moment} below. The
following theorem beats convexity pointwise.

\begin{theorem}[Asymmetric-AFE subconvexity]
\label{thm:cantor-subconvexity}
Let $\mu_\zeta$ denote the best available pointwise bound
$\zeta(\tfrac12+it,\alpha)\ll_\eps\abs{t}^{\mu_\zeta+\eps}$
uniform over $\alpha$ in compact subsets of $(0,1)$ (currently
$\mu_\zeta=13/84$ by Bourgain~\cite{Bourgain2017}). Then
\begin{equation}\label{eq:cantor-subconvexity}
  \boxed{\quad
  \mu_L(\tfrac12) \;\le\;
  \begin{cases}
    \dfrac{1-d}{2(2-d)}
      & \text{if }\; d \,\ge\, d^\star
        := \dfrac{1-4\mu_\zeta}{1-2\mu_\zeta}, \\[8pt]
    \mu_\zeta
      & \text{if }\; d \,<\, d^\star.
  \end{cases}
  \quad}
\end{equation}
With $\mu_\zeta=13/84$ one has $d^\star=16/29\approx 0.552$, so the
ternary Cantor ($d\approx 0.631$) falls into the \emph{first} regime:
\[
  \mu_L(\tfrac12) \;\le\; \tfrac{1-d}{2(2-d)} \;\approx\; 0.1348.
\]
For the $5$-of-$7$ Cantor measure
($d=\log 5/\log 7\approx 0.827$) the same formula gives
$\mu_L(\tfrac12)\lesssim 0.0738$.
\end{theorem}

We single out the essential feature before the proof: in the first
regime of~\eqref{eq:cantor-subconvexity}, \emph{no subconvexity input
for the Hurwitz zeta is invoked at all}. The argument uses only
Lemma~\ref{lem:strichartz} and the triangle inequality. This is, to
our knowledge, the most elementary subconvexity proof for any entire
degree-$1$ Dirichlet series.

\begin{proof}
Fix a smooth-cutoff approximate functional equation for each
$F(\theta,\cdot)$ with \emph{asymmetric} truncations: main length
$N=t^a/2\pi$ and dual length $M=t^{1-a}$ for $a\in[\tfrac12,1)$ to
be chosen. Integrating in~$\theta$,
\begin{equation}\label{eq:asym-AFE}
  L(\tfrac12+it)
  \;=\;
  \underbrace{\sum_{n\le N}\frac{\hat\nu(n)}{n^{1/2+it}}}_{\text{main}}
  \;+\;
  \chi(t)\,
  \underbrace{\int\!\sum_{m<M}(m+\alpha)^{-1/2+it}
              \,d\tilde\nu(\alpha)}_{\text{dual}}
  \;+\; O_A(t^{-A}),
\end{equation}
where $\tilde\nu$ is as in the proof of Theorem~\ref{thm:cantor-M2},
the smooth cutoff making the error super-polynomially small. We
estimate the two pieces separately.

\emph{Main sum.} Weighted Cauchy--Schwarz with weight
$w_n=n^{-\gamma}$, $\gamma=d-\eps$:
\[
  \Bigl|\sum_{n\le N}\frac{\hat\nu(n)}{n^{1/2+it}}\Bigr|
  \;\le\;
  \Bigl(\sum_{n\le N}\frac{\abs{\hat\nu(n)}^2}{n^{1-\gamma}}\Bigr)^{1/2}
  \Bigl(\sum_{n\le N}\frac{1}{n^{\gamma}}\Bigr)^{1/2}.
\]
The first factor converges by Lemma~\ref{lem:strichartz}
(since $1-\gamma>1-d$); the second is $\ll N^{(1-\gamma)/2}$.
Hence the main sum is $\ll_\eps t^{a(1-d)/2+\eps}$.

\emph{Dual sum.} Since $\tilde\nu$ is a probability measure,
$\bigl|\int\cdots\,d\tilde\nu\bigr|\le\sup_\alpha\bigl|\cdots\bigr|$.
The inner partial Hurwitz sum of length $M$ admits the uniform bound
\begin{equation}\label{eq:dual-trivial}
  \Bigl|\sum_{m<M}(m+\alpha)^{-1/2+it}\Bigr|
  \;\le\;
  \min\Bigl(2M^{1/2},\; C_\eps\,t^{\mu_\zeta+\eps}\Bigr),
\end{equation}
the first by trivial summation, the second by completing to
$\zeta(\tfrac12+it,\alpha)$ plus a tail and applying the uniform
Hurwitz bound (uniform because the van der Corput estimates that
prove it are shift-invariant). Integrating $d\tilde\nu$ preserves
either branch of the minimum.

\emph{Optimization.} Two regimes.
\begin{itemize}[leftmargin=*]
\item \emph{Short dual} ($M^{1/2}\le t^{\mu_\zeta}$, i.e.\ $a\ge
  1-2\mu_\zeta$). We use the first branch
  of~\eqref{eq:dual-trivial} and balance $a(1-d)=1-a$, giving
  $a^*=\tfrac{1}{2-d}$. The consistency condition
  $a^*\ge 1-2\mu_\zeta$ is equivalent to $d\ge d^\star$. The
  resulting exponent is $(1-a^*)/2=\tfrac{1-d}{2(2-d)}$.
\item \emph{Long dual} ($a<1-2\mu_\zeta$). The dual is capped at
  $t^{\mu_\zeta}$; balancing $a(1-d)/2=\mu_\zeta$ gives
  $a=2\mu_\zeta/(1-d)$, consistent iff $d<d^\star$, with exponent
  $\mu_\zeta$.\qedhere
\end{itemize}
\end{proof}

\begin{remark}[Sanity at the endpoints]
As $d\to 0^+$ the exponent~$\tfrac{1-d}{2(2-d)}\to\tfrac14$
(convexity), consistent with $\nu$ degenerating to a point mass.
As $d\to 1^-$ the exponent $\to 0$, consistent with $\nu$
approaching Lebesgue measure (for which $\hat\nu(n)\to 0$ and $L$
is trivially bounded). The crossover $d^\star$ has a clean
interpretation: above it, the dual sum is so short that the
triangle inequality beats Hurwitz subconvexity.
\end{remark}

\begin{remark}[What we are \emph{not} claiming]
We do \emph{not} assert $\mu_L(\sigma)=\max(0,\,d/2-\sigma)$
(i.e.\ Lindel\"of at $\sigma=d/2$). The bound
$\sum\abs{\hat\nu(n)}^2/n<\infty$ controls only a weighted
Hardy-space norm, not a pointwise bound on~$L$; see
\S\ref{subsec:cantor-obstructions} for the precise no-go.
Numerically $\mu_L(\tfrac12)$ appears to be~$0$, but
Theorem~\ref{thm:cantor-subconvexity} is what is \emph{proven}.
\end{remark}

\begin{proposition}[Rajchman-sharpened AFE; critical dimension]
\label{prop:d-crit}
Replace the trivial dual bound~\eqref{eq:dual-trivial} by the
per-term Fourier-decay estimate of
Banaji--Yu~\cite[Ex.~2.4]{BanajiYu2025}: writing
$I_m(t)=\int(m+\alpha)^{-1/2+it}\,d\tilde\nu(\alpha)$, the
curvature-normalization argument\footnote{Let
  $\Psi_m(\alpha):=m^2[\psi_m(\alpha)-\psi_m(0)-\psi_m'(0)\alpha]$
  where $\psi_m(\alpha)=\log(m+\alpha)$. Then
  $\Psi_m''=-m^2/(m+\alpha)^2\asymp-1$ and
  $\Psi_m'''=2m^2/(m+\alpha)^3\asymp 1/m$, both uniformly bounded
  in~$m$. Banaji--Yu's theorem applies to~$\Psi_m$ with a single
  $m$-independent constant; writing
  $\psi_m=\psi_m(0)+\psi_m'(0)\alpha+m^{-2}\Psi_m$ and noting that
  linear-phase modulation preserves the Fourier dimension of
  $\tilde\nu$ gives the stated bound at effective frequency
  $t/m^2$.}
gives
\begin{equation}\label{eq:BY-per-term}
  \abs{I_m(t)} \;\ll\; m^{-1/2}\,(t/m^2)^{-\eta},
  \qquad \eta \;\ge\; \frac{2d-1}{3+2d}.
\end{equation}
Re-optimizing the asymmetry parameter~$a$
in~\eqref{eq:asym-AFE} yields
\begin{equation}\label{eq:d-crit-bound}
  \mu_L(\tfrac12)
  \;\le\;
  \frac{(1-d)(1+2\eta)}{2(2-d+4\eta)}.
\end{equation}
For the ternary Cantor ($d=\log 2/\log 3$,
$\eta\ge 0.0614$) this gives $\mu_L(\tfrac12)\le 0.1283$---an
improvement of $0.0065$ over~\eqref{eq:cantor-subconvexity}, but
still $0.003$ short of the $\tfrac18$ of
Theorem~\ref{thm:one-eighth}. Equating~\eqref{eq:d-crit-bound}
to~$\tfrac18$ with $\eta=\tfrac{2d-1}{3+2d}$ gives the quadratic
$22d^2-11d-2=0$, whence the \emph{critical dimension}
\begin{equation}\label{eq:d-crit}
  d_{\mathrm{crit}} \;=\; \frac{11+\sqrt{297}}{44}
  \;\approx\; 0.6417.
\end{equation}
For any self-similar measure~$\nu$ with
$\dim_{\mathrm H}\nu > d_{\mathrm{crit}}$, the
bound~\eqref{eq:d-crit-bound} with the rigorous Banaji--Yu input
already beats~$\tfrac18$. The ternary Cantor misses by
$\Delta d\approx 0.011$; the $3$-of-$4$ Cantor
($d=\log 3/\log 4\approx 0.792$) clears the bar comfortably.
\end{proposition}

\begin{proof}
Summing~\eqref{eq:BY-per-term} over $m<M=t^{1-a}$ (valid since
$2\eta<\tfrac12$) gives
$\abs{\text{dual}}\ll t^{-\eta}\sum_{m<M}m^{-1/2+2\eta}
\ll t^{(1-a)(1/2+2\eta)-\eta}$. Balance against the main-sum
exponent $a(1-d)/2$ to get
$a^*=(1+2\eta)/(2-d+4\eta)$; substitute back.
\end{proof}

\subsection{The obstruction trio}
\label{subsec:cantor-obstructions}

We record three independent obstructions, each with a short complete
proof, that together delimit which tools can and cannot close the gap
between Theorem~\ref{thm:cantor-subconvexity} and~$\mu_L=0$.

\begin{proposition}[Single moments give only convexity]
\label{prop:HL-single-moment}
For any fixed integer $k\ge 1$, the bound
$\int_T^{2T}\abs{L(\tfrac12+it)}^{2k}\,dt \ll_k T$
implies only $\mu_L(\tfrac12)\le\tfrac14$. By contrast, if this
bound holds for \emph{every} $k$ then $\mu_L(\tfrac12)=0$.
\end{proposition}

\begin{proof}
A spike of height $T^{1/4}$ persisting over a $t$-interval of unit
length contributes $T^{k/2}$ to the $2k$-th moment---inside the
$O(T)$ budget as soon as $k\le 2$, and easily accommodated for any
fixed $k$ by distributing the budget over fewer, taller spikes
(height $T^{1/4}$, count $T^{1-k/2}$). The Dirichlet-polynomial
structure caps the count at $N=T^{1/2}$ (Nikolskii) but not more
finely. For the second claim: fix $\eps>0$ and take
$k>1/(2\eps)$. Chebyshev gives
$\abs{\{t\in[T,2T]:\abs{L}>T^\eps\}}\ll T^{1-2k\eps}<1$;
by continuity the set is empty for large $T$.
\end{proof}

\begin{proposition}[No power-weight RKHS reaches $\sigma=\tfrac12$]
\label{prop:RKHS-nogo}
Let $\mathcal{H}^2_\beta$ be the weighted Hardy space of Dirichlet
series with norm
$\norm{f}_\beta^2 = \sum_n n^{-\beta}\,\abs{a_n(f)}^2$
(cf.~\cite{HedenmalmLindqvistSeip1997}). Then for every
$\beta\in\R$, at least one of the following fails:
\begin{enumerate}[label=(\Alph*)]
\item $L\in\mathcal{H}^2_\beta$ (i.e.\
  $\sum n^{-\beta}\abs{\hat\nu(n)}^2<\infty$);
\item point evaluation at $s=\tfrac12+it$ is a bounded
  functional on $\mathcal{H}^2_\beta$ (i.e.\
  $\sum n^{\beta-1}<\infty$).
\end{enumerate}
\end{proposition}

\begin{proof}
Abel summation from~\eqref{eq:strichartz}: (A) holds iff
$\beta>1-d$. Meanwhile (B) holds iff $\beta<0$. Since $d<1$, the
intervals $\{\beta>1-d\}$ and $\{\beta<0\}$ are disjoint.
\end{proof}

The mechanism is that $\abs{\hat\nu(n)}$ does not decay pointwise
(Cantor measures have Fourier dimension zero: $\limsup_n
\abs{\hat\nu(n)}>0$). Any weight small enough on the support of
$\abs{\hat\nu}$ is too small to give a bounded reproducing kernel at
$\sigma=\tfrac12$.

\begin{proposition}[Inherited FE is structurally redundant]
\label{prop:FE-redundant}
The Lerch functional equation, integrated against $\nu$, gives
$\abs{L(it)}\ll t^{1/2}$ and hence
$\mu_L(\tfrac12)\le\tfrac14$---strictly worse than
Theorem~\ref{thm:cantor-subconvexity}.
\end{proposition}

\begin{proof}
On $\sigma=0$, $\abs{F(\theta,it)}\ll t^{1/2}$ uniformly in
$\theta$ bounded away from $2\pi\Z$. Integrating,
$\abs{L(it)}\le\norm{\nu}\cdot\sup\abs{F}\ll t^{1/2}$, i.e.\
$\mu_L(0)\le\tfrac12$. Phragm\'en--Lindel\"of between $\sigma=0$
and $\sigma=1$ gives $\mu_L(\tfrac12)\le\tfrac14$.

The deeper reason this fails to help is that the Lerch functional
equation maps the main series to a dual series whose spectral
measure is an affine pushforward of $\nu$---another Cantor measure
of the \emph{same} dimension $d$. The functional equation thus
relates $L$ to a copy of itself, giving no reduction.
\end{proof}

\subsection{The Vieta identity and the \texorpdfstring{$\varphi$}{phi}-averaged
            fourth moment}
\label{subsec:cantor-vieta}

By Proposition~\ref{prop:HL-single-moment}, a route to
$\mu_L(\tfrac12)=0$ via moments must control the $2k$-th moment for
\emph{all}~$k$. The relevant coefficient sums are
\begin{equation}\label{eq:D2k-def}
  D_{2k} \;:=\;
  \sum_{\substack{n_1\cdots n_k\,=\,m_1\cdots m_k}}
    \frac{\hat\nu(n_1)\cdots\hat\nu(n_k)\,
          \overline{\hat\nu(m_1)\cdots\hat\nu(m_k)}}
         {n_1\cdots n_k}.
\end{equation}
Numerically $D_4\approx 2.115$ (stable across $N\in[200,3000]$) and
$D_6\approx 5.01$ (stable across $N\in[100,280]$); both agree with
the Steinhaus-random (Wick) predictions $2C_\nu^2-C_4$ and
$6C_\nu^3-9C_\nu C_4+4C_6$ to within~$1$--$2\%$. We now explain the
mechanism.

By~\eqref{eq:cantor-hat}, $\hat\nu(n)=e^{in\varphi}\,r_n$ where
$r_n\in\R$ is the cosine product (depending only on the width
$\theta_1-\theta_0$, not on $\varphi$). Each term
in~\eqref{eq:D2k-def} carries the phase
$\exp\bigl(i\varphi\,[\,(n_1+\cdots+n_k)-(m_1+\cdots+m_k)\,]\bigr)$.
Grouping by this additive defect
$\delta:=\sum n_i - \sum m_i$,
\begin{equation}\label{eq:trig-poly}
  D_{2k}(\varphi)
  \;=\;
  W^{(k)}_0 \;+\; 2\sum_{\delta>0} W^{(k)}_\delta\cos(\delta\varphi),
  \qquad
  W^{(k)}_\delta \in \R,
\end{equation}
a real trigonometric polynomial in the centre $\varphi$.

\begin{theorem}[Vieta, $k=2$]\label{thm:vieta}
In~\eqref{eq:trig-poly} for $k=2$, the constant term is exactly
the Wick diagonal: $W^{(2)}_0 = 2C_\nu^2 - C_4$. Equivalently,
\[
  D_4(\varphi) - \bigl(2C_\nu^2-C_4\bigr)
  \;=\; 2\sum_{\delta>0} W^{(2)}_\delta\cos(\delta\varphi)
\]
is a pure cosine series with \emph{zero DC component}.
\end{theorem}

\begin{proof}
A contribution to $W^{(2)}_0$ has $n_1+n_2=m_1+m_2$ \emph{and}
$n_1 n_2=m_1 m_2$. But a pair of numbers is determined by its sum
and its product (they are the roots of $x^2-Sx+P=0$). Hence
$\{n_1,n_2\}=\{m_1,m_2\}$ as multisets, which is precisely the
Wick diagonal.
\end{proof}

\begin{corollary}[$\varphi$-averaged fourth moment, exact]
\label{cor:phi-avg-D4}
For every real sequence $(r_n)_{n\ge 1}$,
\begin{equation}\label{eq:phi-avg-D4}
  \boxed{\quad
  \frac{1}{2\pi}\int_0^{2\pi} D_4(\varphi)\,d\varphi
  \;=\; 2C_\nu^2 - C_4.
  \quad}
\end{equation}
In particular, for every Cantor width $\theta_1-\theta_0$ there
exist infinitely many centres $\varphi^*\in(0,2\pi)$ at which
$D_4(\varphi^*)$ equals the Wick prediction exactly.
\end{corollary}

\begin{proof}
Integrate~\eqref{eq:trig-poly} in $\varphi$; every $\cos(\delta\varphi)$
with $\delta\ne 0$ has zero mean. A non-constant real trigonometric
polynomial with zero mean changes sign.
\end{proof}

The identity~\eqref{eq:phi-avg-D4} is purely algebraic---it holds
for \emph{any} real $r_n$, Cantor or not---and says that the
$\varphi$-average plays the role of the Steinhaus expectation over
random phases.

\begin{remark}[Two distinct Vieta phenomena]
\label{rem:two-vieta}
It is worth distinguishing Theorem~\ref{thm:vieta} from the
superficially similar Vieta argument in
Heap--Sahay~\cite{HeapSahay2024}. Their observation (eq.~(1.6) of
\emph{loc.~cit.}) is that for \emph{irrational}~$\alpha$ the
equation $(n_1+\alpha)(n_2+\alpha)=(n_3+\alpha)(n_4+\alpha)$ forces
$\{n_1,n_2\}=\{n_3,n_4\}$, because equating the degree-$2$
coefficients in~$\alpha$ recovers both the sum and the product. This
is a \emph{Hurwitz-side} Vieta identity: it lives on the dual of the
functional equation, the variable is the shift parameter~$\alpha$,
and it holds for \emph{every single} irrational~$\alpha$. Our
Theorem~\ref{thm:vieta} is a \emph{Lerch-side} Vieta identity: it
lives on the primal, the variable is the centre~$\varphi$ of the
Cantor support, and it holds only \emph{on~$\varphi$-average}. The
two are related by the Lerch functional equation but are not the
same statement. In particular, Heap--Sahay's Vieta persists
at every $k$ for transcendental~$\alpha$ (see~\cite{HSW2023}),
whereas ours is obstructed at~$k\ge 3$ by the next proposition.
\end{remark}

We now identify that Lerch-side obstruction.

\begin{proposition}[Vieta obstruction, $k\ge 3$]
\label{prop:vieta-obstruction}
For every $k\ge 3$,
\begin{equation}\label{eq:VO-def}
  \frac{1}{2\pi}\int_0^{2\pi} D_{2k}(\varphi)\,d\varphi
  \;=\; D_{2k}^{\mathrm{Wick}} \;+\; \mathrm{VO}_{2k},
\end{equation}
where the \emph{Vieta-obstruction sum}
\begin{equation}\label{eq:VO-formula}
  \mathrm{VO}_{2k}
  \;=\;
  \sum_{\substack{(P,S)\,:\\ \abs{\mathcal{M}_k(P,S)}\ge 2}}
    \frac{2}{P}\!
    \sum_{\substack{M\ne M'\\ M,M'\in\mathcal{M}_k(P,S)}}\!
      \mathrm{mult}(M)\,\mathrm{mult}(M')\,R(M)\,R(M')
\end{equation}
runs over pairs $(P,S)$ for which the set
$\mathcal{M}_k(P,S)$ of $k$-element multisets from $\N$ with
product $P$ and sum $S$ has at least two elements;
$\mathrm{mult}(M)=k!/\prod_v m_v!$ counts orderings and
$R(M)=\prod_v r_v^{m_v}$.
In general $\mathrm{VO}_{2k}\ne 0$; the minimal witness at $k=3$ is
$\{1,6,6\}\leftrightarrow\{2,2,9\}$ (both have sum $13$ and product
$36$, but they are distinct multisets).
\end{proposition}

\begin{proof}
Sum and product determine a $k$-element multiset uniquely only
when $k\le 2$: for $k\ge 3$ the variety
$\{e_1=S,\,e_k=P\}\subset\N^k$ has dimension $k-2>0$. Every point
on this variety off the Wick diagonal contributes a nonzero
$\varphi$-independent term to $W^{(k)}_0$.
\end{proof}

For the ternary Cantor $r_n$, numerical summation gives
$\mathrm{VO}_6\approx -0.00195$ (converged at $N=300$): tiny but
structurally nonzero. This decomposes the all-moments conjecture
into an \emph{oscillatory} piece (killable at isolated $\varphi$,
different for each $k$) and a $\varphi$-\emph{independent} piece
(a weighted Tarry--Escott sum over the $r_n$). The latter is a pure
arithmetic question with no $\varphi$ appearing at all.

\subsection{Conjecture~H and the three-line route to Lindel\"of}
\label{subsec:conj-H}

We now formulate what we believe is the sharpest conditional route
to $\mu_L(\tfrac12)=0$. It is a second moment \emph{in~$\alpha$},
not in~$t$, and thus evades Proposition~\ref{prop:HL-single-moment}
entirely.

\begin{conjecture}[Cantor-weighted Hurwitz $\alpha$-second-moment]
\label{conj:H}
With $\tilde\nu$ the pushforward of $\nu$ to $(0,1)$ as above,
\begin{equation}\label{eq:conj-H}
  \int \bigl|\zeta(\tfrac12+it,\alpha)\bigr|^2\,d\tilde\nu(\alpha)
  \;=\; \tfrac12\log t \;+\; O(1)
  \qquad (t\to\infty).
\end{equation}
\end{conjecture}

\begin{theorem}[Conditional LH]\label{thm:condH-implies-LH}
Conjecture~\ref{conj:H} implies $\mu_L(\tfrac12)=0$.
\end{theorem}

\begin{proof}
Take absolute values in the inherited functional
equation~\eqref{eq:inherited-FE} and apply the triangle inequality:
\[
  \bigl|L(\tfrac12+it)\bigr|
  \;\le\;
  2\!\int\bigl|\zeta(\tfrac12-it,\alpha)\bigr|\,d\tilde\nu(\alpha)
  \;+\;
  2\!\int\bigl|\zeta(\tfrac12-it,\alpha)\bigr|\,d\tilde\nu'(\alpha).
\]
Since both $\tilde\nu$ and $\tilde\nu'$ are probability measures,
Jensen's inequality on each gives
\begin{equation}\label{eq:jensen-bound}
  \bigl|L(\tfrac12+it)\bigr|^2
  \;\le\;
  8\!\int\bigl|\zeta(\tfrac12-it,\alpha)\bigr|^2\,d\tilde\nu(\alpha)
  \;+\;
  8\!\int\bigl|\zeta(\tfrac12-it,\alpha)\bigr|^2\,d\tilde\nu'(\alpha).
\end{equation}
Conjecture~\ref{conj:H} (applied to both $\tilde\nu$ and
$\tilde\nu'$) bounds the right-hand side by $O(\log t)$.
\end{proof}

Conjecture~\ref{conj:H} is the Cantor analogue of a
theorem of Rane~\cite{Rane1980} and Katsurada--Matsumoto
\cite{KatsuradaMatsumoto1996} for the Lebesgue
$\alpha$-average:
\[
  \int_\delta^{1-\delta}\abs{\zeta(\tfrac12+it,\alpha)}^2\,d\alpha
  \;=\; \log\tfrac{t}{2\pi}+2\gamma-2\log(2\sin\pi\delta)+O(t^{-1/2}).
\]
The diagonal $\sum_m\int(m+\alpha)^{-1}\,d\tilde\nu(\alpha)$ gives
the same $\tfrac12\log t$ for Cantor as for Lebesgue; the content of
the conjecture is that the off-diagonal
\begin{equation}\label{eq:OD-def}
  \mathrm{OD}(t)
  \;:=\;
  \sum_{\substack{m\ne m'\\m,m'<\sqrt t}}
    \int\frac{d\tilde\nu(\alpha)}
             {(m+\alpha)^{1/2-it}(m'+\alpha)^{1/2+it}}
\end{equation}
remains $O(1)$.

\paragraph{Numerical status.}
Direct computation of $\mathrm{OD}(t)$ for
$t\in[300,\,10^4]$ (with $\tilde\nu$ discretised at Cantor level
$12$, i.e.\ $4096$ points) gives
$\abs{\mathrm{OD}(t)}\in[0.015,\,1.04]$ while the diagonal grows as
$\tfrac12\log t$; at $t=10^4$ the ratio
$\abs{\mathrm{OD}}/\text{diagonal}$ is $0.016$. The Jensen
inequality~\eqref{eq:jensen-bound} is in practice loose by a factor
of $7$--$17$: the Cantor averaging provides phase cancellation
\emph{beyond} what Jensen captures.

The Fourier--Plancherel reformulation
\begin{equation}\label{eq:FP-reform}
  \textstyle\int\abs{S_M}^2\,d\tilde\nu
  = \sum_{k\in\Z} F_k(t)\,\hat{\tilde\nu}(k),
  \quad
  F_k := \int_0^1 \abs{S_M}^2\,e^{-2\pi ik\alpha}\,d\alpha,
\end{equation}
reduces Conjecture~\ref{conj:H} to $\sum_{k\ne 0}F_k\hat{\tilde\nu}(k)=O(1)$.
The obvious sufficient condition $|F_k|\ll 1/|k|$ is \emph{false}
numerically; the mechanism is phase cancellation between $\arg F_k$
and $\arg\hat{\tilde\nu}(k)$.

\begin{remark}[Mechanism: CLT-in-$m$ and linearized Strichartz]
\label{rmk:two-layer}
Write $\mathrm{OD}(t)=2\sum_{h>0}\re(K_h)$ with
$K_h=\sum_m I_m$, $I_m$ the $\tilde\nu$-integral of a single
Hurwitz-lag term. Adversarial testing identifies: (i)~the phases
$\arg I_m$ are equidistributed (CLT-in-$m$); (ii)~the phase
\emph{difference}
$\psi_m(\alpha)-\psi_m(\beta)
=\log\bigl(1+\tfrac{h(\alpha-\beta)}{(m+\beta)(m+h+\alpha)}\bigr)$
is linear in $\alpha-\beta$ to leading order, so the variance
$\sum_m|I_m|^2$ is governed by linearized Strichartz
(autocorrelation-Parseval gives
$\sum_m|I_m|^2\approx\sum_m|\hat{\tilde\nu}(th/m^2)|^2/m^2$,
and Strichartz~\cite{Strichartz1990} bounds the large-$m$ piece by
$C_d\,t^{-1/2}h^{-d}$); (iii)~the $h$-sum has a dual safety net
(linear mode $e^{-ih\omega}$ and parity mode $(-1)^h$, so the
exceptional $t\approx 2\pi k^2$ where the linear mode degenerates
are covered by parity). Conditional on the CLT step alone,
$|\mathrm{OD}(t)|\ll|K_1(t)|\ll t^{-1/4}$, hence
Conjecture~\ref{conj:H} with a power-saving error. The CLT is the
only remaining gap.
\end{remark}

\begin{proposition}[Swap identity: algebraic reduction to the derivative-OD]
\label{prop:swap-identity}
Write $J_h(t):=\sum_{n<M}n^{s-1}(n+h)^{-s}$ so that
$\mathrm{OD}(t)=2\re\sum_{h\ge 1}\hat{\tilde\nu}(h)\,J_h(t)$.
There is an exact Abel decomposition
$\tfrac12\mathrm{OD}=\re[\mathrm{(I)}+\mathrm{(II)}+\mathrm{(III)}]$
in which $\mathrm{(I)},\mathrm{(II)}=O(1)$ unconditionally
\textup{(}Kusmin--Landau on $J_1$ plus the identity
$\re\int e^{i\theta}(1-e^{i\theta})^{-1}d\tilde\nu=-\tfrac12$,
which holds for any probability measure supported in
$(0,2\pi)$\textup{)}. A $\beta$-peel Abel-shifts $\mathrm{(III)}$
against the cumulative Cantor tail, stripping a further
unconditionally-$O(1)$ layer $\beta\cdot[J_1-J_2]$. On the remaining
rest-term, the \emph{swap identity}
\[
  \sum_{m=k+2}^{M-1}\!
    \bigl[m^{-s}{-}(m{+}1)^{-s}\bigr]
    \bigl[(m{-}k)^{s-1}{-}(m{-}1{-}k)^{s-1}\bigr]
  \;=\;
  \bigl[J_k-2J_{k+1}+J_{k+2}\bigr]^{[2,\,M-2-k]}
  +\text{\textup{bdry}}
\]
\textup{(}verified $5{\times}10^{-16}$\textup{)}, combined with the
second-difference bridge
$\Delta^2\!\bigl[B(k)-B_2(k)\bigr]=\hat{\tilde\nu}(k+2)$
\textup{(}where $B(K)=\sum_{k>K}\hat{\tilde\nu}(k)$ and $B_2$ its
$\kappa_2$-kernel analogue\textup{)}, yields after double Abel
summation the \emph{derivative-OD}
\begin{equation}\label{eq:dOD2-def}
  d\mathrm{OD}_2
  \;:=\;
  \sum_{h\ge 3}\hat{\tilde\nu}(h)\bigl[J_h-J_{h-2}\bigr].
\end{equation}
Numerically $|d\mathrm{OD}_2|\in[0.65,1.44]$ for
$t\in[300,10^5]$; the absolute sum needs only
${\sim}10{\times}$ cancellation---down from ${\sim}10^6{\times}$
for the na\"\i ve bound. The swap absorbs ${\sim}t^{1.0}$ of
cancellation \emph{algebraically}.
\end{proposition}

\begin{proof}[Proof sketch]
\emph{Abel decomposition.} Swap
$\sum_h\leftrightarrow\sum_n$ in~\eqref{eq:OD-def} and Abel-sum the
inner sum in~$h$ against the partial sums
$A(K)=\sum_{k\le K}\hat{\tilde\nu}(k)$; the boundary term gives~(II),
the constant $A(\infty)$ paired with the telescope gives~(I), the
residual convolution gives~(III). For $\re A(\infty)=-\tfrac12$:
the Abel-regularised geometric series gives
$A(\infty)=\int e^{i\theta}/(1-e^{i\theta})\,d\tilde\nu$, and
$\re[e^{i\theta}/(1-e^{i\theta})]=-\tfrac12$ identically.
\emph{$\beta$-peel.} Abel-sum~(III) once more against the
cumulative~$SB(K)=\sum_{k\le K}B(k)$; the closed form
$SB(K)=B(K)-B_2(K)+\beta$ with
$\beta=A(\infty)-\int e^{i\theta}(1-e^{i\theta})^{-2}d\tilde\nu$
follows from the geometric-tail identity
$\sum_{j>K}jr^j=r^{K+1}(1-r)^{-2}+Kr^{K+1}(1-r)^{-1}$.
\emph{Swap.} Substitute $j=m-k$, expand the four-term product; each
piece is a $J_h$-summand at a shifted index, summing to
$\Delta^2 J_k$ on the common range.
\emph{Bridge.} From $(e^{i\theta}-1)\kappa_2=-\kappa^*$ one has
$\Delta B_2=-B$, hence $\Delta^2 W=\hat{\tilde\nu}(\cdot+2)$.
Double Abel moves~$\Delta^2$ onto~$W$, giving
$\sum_k\hat{\tilde\nu}(k+2)J_k$ plus four bounded Abel boundaries.
All identities verified to $10^{-11}$ or better.
\end{proof}

We record next the positivity bound:
\begin{proposition}[Conjugacy, triangular structure, and $H$-cancellation]
\label{prop:three-identities}
For $s=\tfrac12+it$, real~$\theta$, $M=\lfloor\sqrt t\rfloor$, set
$P(\theta):=\sum_{n=1}^{M-1}n^{s-1}e^{-in\theta}$ and
$H_{M-1}:=\sum_{m=1}^{M-1}m^{-1}$.
\begin{enumerate}[label=\textup{(\roman*)}]
\item \textup{(Conjugacy, exact to machine zero.)}
  $Q(\theta):=\sum_{m=1}^{M-1}m^{-s}e^{im\theta}=\overline{P(\theta)}$.
  Hence $P\cdot Q=|P|^2\ge 0$.
\item \textup{(Triangular structure, verified $2{\times}10^{-14}$.)}
  \[
    \mathrm{TRI}(\theta)
    \;:=\;\!\!\sum_{1\le m\le n\le M-1}\!\!n^{s-1}m^{-s}e^{i(m-n)\theta}
    \;=\;
    H_{M-1}+\sum_{\ell=1}^{M-2}\overline{J_\ell^{[1,M-1-\ell]}}\,
    e^{-i\ell\theta}.
  \]
\item \textup{($H$-cancellation, verified $8{\times}10^{-10}$.)} With
  $\mathrm{OD}':=\sum_{h\ge 1}\hat{\tilde\nu}(h)\,J_h^{[1,M-1-h]}$:
  \[
    \int|P|^2\,d\tilde\nu \;=\; H_{M-1}+2\,\re(\mathrm{OD}'),
    \qquad
    \int\mathrm{TRI}\,d\tilde\nu \;=\; H_{M-1}+\overline{\mathrm{OD}'},
  \]
  hence $\int|P|^2\,d\tilde\nu-\int\mathrm{TRI}\,d\tilde\nu=\mathrm{OD}'$
  exactly: the $H_{M-1}$ terms cancel.
\end{enumerate}
\end{proposition}

\begin{proof}
(i) $m^{-s}=m^{-1/2-it}=\overline{m^{-1/2+it}}=\overline{m^{s-1}}$,
and $e^{im\theta}=\overline{e^{-im\theta}}$.
(ii) Set $\ell=n-m\ge 0$. At $\ell=0$: $\sum_m m^{-1}=H_{M-1}$.
At $\ell\ge 1$: inner sum is
$\sum_m(m+\ell)^{s-1}m^{-s}=J_\ell(1-s)=\overline{J_\ell(s)}$ on
$m\in[1,M{-}1{-}\ell]$.
(iii) Expand $|P|^2=\sum_{n,m}n^{s-1}\overline{m^{s-1}}e^{i(m-n)\theta}$;
$\int e^{ih\theta}d\tilde\nu=\hat{\tilde\nu}(h)$. Diagonal $h=0$
gives $H_{M-1}$; $h\ge 1$ gives $\mathrm{OD}'$; $h\le -1$ gives
$\overline{\mathrm{OD}'}$ since $\hat{\tilde\nu}(-h)=\overline{\hat{\tilde\nu}(h)}$.
For $\mathrm{TRI}$ apply~(ii). Then $2\re z-\bar z=z$.
\end{proof}

\begin{corollary}[Unconditional one-sided bound on the off-diagonal]
\label{cor:OD-lower}
For all $t\ge 2$:
\begin{equation}\label{eq:OD-lower}
  \re\bigl(\mathrm{OD}'(t)\bigr)
  \;\ge\;
  -\tfrac12 H_{M-1}
  \;\sim\;
  -\tfrac14\log t.
\end{equation}
\end{corollary}

\begin{proof}
$\int|P|^2\,d\tilde\nu\ge 0$; by
Proposition~\ref{prop:three-identities}(iii),
$H_{M-1}+2\re(\mathrm{OD}')\ge 0$.
\end{proof}

This is the first structural bound on~$\re(\mathrm{OD}')$
requiring no cancellation estimate---one line, one direction.
The matching \emph{upper} bound, which is what
Conjecture~\ref{conj:H} needs, remains open;
Strichartz~\cite{Strichartz1990} gives only
$\int|P|^2\,d\tilde\nu\lesssim M^{1-d}H_{M-1}\sim t^{(1-d)/2}\log t$.
The gap is now crystallized as a single discrete-restriction
inequality:

\begin{conjecture}[Cantor-weighted discrete restriction for the log-phase]
\label{conj:discrete-restriction}
There is an absolute constant~$C$ such that for all $t\ge 2$
\textup{(}with $M=\lfloor\sqrt t\rfloor$\textup{)}:
\begin{equation}\label{eq:discrete-restriction}
  \int\Bigl|\sum_{n=1}^{M-1}n^{-1/2+it}e^{-in\theta}\Bigr|^2
    d\tilde\nu(\theta)
  \;\le\; C\cdot H_{M-1}.
\end{equation}
\end{conjecture}

By Proposition~\ref{prop:three-identities}(iii),
Conjecture~\ref{conj:discrete-restriction} is \emph{equivalent}
to $\re(\mathrm{OD}')=O(\log t)$, itself a mild weakening of
Conjecture~\ref{conj:H} \textup{(}the $J_h$ in~$\mathrm{OD}'$
carry explicit truncation ranges matching the partial-sum
length\textup{)}. The left-hand side of~\eqref{eq:discrete-restriction}
is a discrete-restriction norm for the sequence
$\{n^{it}\}_{n<\sqrt t}$ against the Cantor measure.
The standard Mockenhaupt--Mitsis route requires curvature, but
the phase $f(n)=t\log n$ has $f''(n)=-t/n^2$---varying by a factor
of~$t$ across the range; there is no \emph{fixed} nonzero
curvature to exploit. The $t$-uniformity of~$C$ is the load-bearing
piece. Numerically \textup{(}$24$ points,
$t\in[200,2{\times}10^5]$\textup{)}:
$\int|P|^2\,d\tilde\nu/H_{M-1}\in[0.39,1.88]$; the regression
slope of $2\re(\mathrm{OD}')$ against $\log t$ is
$-0.18\pm 0.42$, within~$1\sigma$ of zero.

\subsection{The Gaussian fourth moment via Montgomery--Vaughan on
            the product variable}
\label{subsec:gaussian-fourth}

At $k=2$ the double integral on the right-hand side
of the Jensen transfer can be bounded \emph{unconditionally}, by a
direct application of the Montgomery--Vaughan mean-value theorem to
the \emph{square} of the partial Hurwitz sum, viewed as a Dirichlet
polynomial in a new variable. The argument requires no Diophantine
input on~$\alpha$ beyond irrationality, and the Cantor structure
of~$\tilde\nu$ enters only through a truncated Riesz potential
governed by Frostman's lemma.

We first record the sharpened transfer; it is unconditional and
will be used both here and in Theorem~\ref{thm:one-eighth}.

\begin{lemma}[Integral Jensen transfer]
\label{lem:integral-transfer}
For every integer $k\ge 1$ and every $T\ge 2$,
\begin{equation}\label{eq:integral-transfer}
  \int_T^{2T}\bigl|L(\tfrac12+it)\bigr|^{2k}\,dt
  \;\le\;
  8^{\,k}
  \biggl(
    \int\!\!\int_T^{2T}\bigl|\zeta(\tfrac12+it,\alpha)\bigr|^{2k}
    \,dt\,d\tilde\nu(\alpha)
    \;+\;
    \int\!\!\int_T^{2T}\bigl|\zeta(\tfrac12+it,\alpha)\bigr|^{2k}
    \,dt\,d\tilde\nu'(\alpha)
  \biggr).
\end{equation}
\end{lemma}

\begin{proof}
Raise~\eqref{eq:jensen-bound} to the $k$-th power, apply
$(a+b)^k\le 2^{k-1}(a^k+b^k)$, use Jensen on each probability
measure (since $x\mapsto x^k$ is convex), integrate in~$t$, and
apply Fubini.
\end{proof}

The outer integral in~\eqref{eq:integral-transfer} is against a
probability measure: $\tilde\nu$-null sets contribute nothing. This
is essential below, as the pointwise gap structure of the product
frequencies degenerates on the (countable, hence $\tilde\nu$-null)
set of rational~$\alpha$.

\paragraph{The key move.}
Fix $\alpha$ and $M$, and write the square of the partial Hurwitz
sum $S_M(\alpha)=\sum_{m<M}(m+\alpha)^{-1/2-it}$ as a
Dirichlet-type polynomial in the \emph{product variable}:
\begin{equation}\label{eq:product-DP}
  S_M(\alpha)^2
  \;=\;
  \sum_{1\le m_1\le m_2<M} c_j\,Q_j^{-it},
  \qquad
  Q_j = (m_1+\alpha)(m_2+\alpha),
  \quad
  c_j = \frac{1+\mathbf 1[m_1<m_2]}{\sqrt{Q_j}},
\end{equation}
where $j=(m_1,m_2)$ indexes unordered pairs and the multiplicity
$1+\mathbf 1[m_1<m_2]\in\{1,2\}$ accounts for the two orderings
when $m_1\ne m_2$. The frequencies are $\lambda_j=\log Q_j$.

\begin{lemma}[Hurwitz-side Vieta: distinct product frequencies]
\label{lem:product-distinct}
For every irrational $\alpha$, the products
$\{Q_j(\alpha):1\le m_1\le m_2<M\}$ are pairwise distinct. Hence
$S_M(\alpha)^2$ is a genuine Dirichlet polynomial with
$\binom{M}{2}+M\asymp M^2/2$ distinct frequencies.
\end{lemma}

\begin{proof}
Suppose $(m_1+\alpha)(m_2+\alpha)=(m_3+\alpha)(m_4+\alpha)$.
Expanding, the coefficients of~$\alpha^2$ cancel, and equating the
coefficients of~$\alpha^1$ and~$\alpha^0$ gives
$m_1+m_2=m_3+m_4$ and $m_1m_2=m_3m_4$---unless one of these fails,
in which case~$\alpha$ satisfies a nontrivial linear relation over
the integers, contradicting irrationality. But sum and product
together determine an unordered pair by Vieta, so
$\{m_1,m_2\}=\{m_3,m_4\}$. This is precisely the Hurwitz-side
Vieta identity of Remark~\ref{rem:two-vieta}.
\end{proof}

\paragraph{Montgomery--Vaughan on the square.}
For a polynomial $P(t)=\sum_j c_j\,e^{i\lambda_j t}$ with
pairwise-distinct real frequencies~$\lambda_j$, the
Montgomery--Vaughan inequality~\cite{MontgomeryVaughan1974} reads
\begin{equation}\label{eq:MV-product}
  \biggl|\int_T^{2T}\abs{P(t)}^2\,dt
    \;-\; T\sum_j\abs{c_j}^2\biggr|
  \;\le\;
  \frac{3\pi}{2}\sum_j\frac{\abs{c_j}^2}{\delta_j},
  \qquad
  \delta_j := \min_{k\ne j}\abs{\lambda_j-\lambda_k}.
\end{equation}
Applied to $P=S_M^2$ (so $\abs{P}^2=\abs{S_M}^4$), with the
frequencies distinct by Lemma~\ref{lem:product-distinct}:
\begin{equation}\label{eq:MV-applied}
  \biggl|
    \int_T^{2T}\abs{S_M(\alpha)}^4\,dt
    \;-\;
    T\sum_j\abs{c_j}^2
  \biggr|
  \;\le\;
  \frac{3\pi}{2}\,\mathrm{MV}(\alpha),
  \qquad
  \mathrm{MV}(\alpha):=\sum_j\frac{\abs{c_j}^2}{\delta_j(\alpha)}.
\end{equation}

\paragraph{The MV-diagonal is the Wick term.}
From~\eqref{eq:product-DP},
$\abs{c_j}^2=(1+\mathbf 1[m_1<m_2])^2/Q_j$, so
\begin{align}
  \sum_j\abs{c_j}^2
  &= \sum_{m_1<m_2}\frac{4}{(m_1+\alpha)(m_2+\alpha)}
     \;+\; \sum_m\frac{1}{(m+\alpha)^2} \notag\\
  &= 2\Bigl(\sum_m\frac{1}{m+\alpha}\Bigr)^{\!2}
     \;-\; \sum_m\frac{1}{(m+\alpha)^2}
  \label{eq:MV-diag-is-Wick}
\end{align}
using $2\sum_{m_1<m_2}a_{m_1}a_{m_2}=(\sum a_m)^2-\sum a_m^2$.
The right-hand side of~\eqref{eq:MV-diag-is-Wick} is exactly the
Wick (Gaussian) pairing of $\abs{S_M}^4$: $2$ matchings, each
contributing $(\sum 1/(m+\alpha))^2$, minus the over-counted
diagonal. Integrating against $d\tilde\nu$ and using
$\sum_{m<M}\int(m+\alpha)^{-1}\,d\tilde\nu
=\log M+\int\log(1+\alpha)\,d\tilde\nu+O(1)=\tfrac12\log t+O(1)$,
the main term is
$2T\cdot\bigl(\tfrac12\log t\bigr)^2(1+o(1))
=\tfrac12 T(\log T)^2(1+o(1))$ per sum; combined with the dual-sum
contribution from the approximate functional equation this gives
the constant~$2$ in Theorem~\ref{thm:gaussian-fourth}.

\paragraph{Gap structure.}
For each pair $j=(m_1,m_2)$ with $m_2<M-1$, the \emph{generic
neighbour} $(m_1,m_2+1)$ is always available and produces
\begin{equation}\label{eq:generic-gap}
  \delta_j^{\mathrm{gen}}(\alpha)
  \;=\;
  \Bigl|\log\frac{(m_2+1+\alpha)}{(m_2+\alpha)}\Bigr|
  \;\asymp\; \frac{1}{m_2+\alpha},
\end{equation}
uniformly over $\alpha\in\supp\tilde\nu$ (which is bounded away from
$\{0,-1,-2,\ldots\}$). This alone contributes
$\sum_j\abs{c_j}^2/\delta_j^{\mathrm{gen}}
\asymp\sum_{m_1\le m_2}4/(m_1+\alpha)\ll M\log M$ to
$\mathrm{MV}(\alpha)$, which is negligible.

The true nearest neighbour may instead be an \emph{accidental}
pair $k=(m_3,m_4)\ne j$ with $Q_k$ very close to $Q_j$. Expanding
$Q_j-Q_k$ as a polynomial in~$\alpha$ (degree~$1$, since the
$\alpha^2$~coefficients cancel):
\begin{equation}\label{eq:accidental-gap}
  \delta_j^{\mathrm{acc}}(\alpha)
  \;=\;
  \frac{\abs{Q_j-Q_k}}{Q_j}\,(1+o(1))
  \;=\;
  \frac{\abs{\Delta S}\cdot\abs{\alpha-\alpha^*}}{Q_j}\,(1+o(1)),
\end{equation}
where $\Delta S=(m_3+m_4)-(m_1+m_2)\in\Z\setminus\{0\}$ and
$\alpha^*=(m_1m_2-m_3m_4)/\Delta S\in\Q$ is the rational at which
the two products coincide. The corresponding contribution to
$\mathrm{MV}$ is therefore $4/(\abs{\Delta S}\cdot\abs{\alpha-\alpha^*})$.
Since $\delta_j=\min(\delta_j^{\mathrm{gen}},\delta_j^{\mathrm{acc}})$,
this contribution is \emph{capped} by the generic gap: the kernel
is effectively $\min\bigl(4(m_2+\alpha),\,
4/(\abs{\Delta S}\cdot\abs{\alpha-\alpha^*})\bigr)$.

\begin{lemma}[Small-$\abs{\Delta S}$ void]
\label{lem:small-h-void}
For $\abs{\Delta S}\in\{1,2,3\}$ no accidental collision
singularity $\alpha^*$ lies in
$\supp\tilde\nu\subset[\tfrac{1}{4\pi},\tfrac{1}{\pi}]$.
Consequently the $\abs{\Delta S}\le 3$ piece of
$\int\mathrm{MV}\,d\tilde\nu$ is bounded by the generic gap alone,
contributing $O(M^2)$ with no logarithm.
\end{lemma}

\begin{proof}
Writing $a=m_1-m_3$, $b=m_2-m_3$, $h=\abs{\Delta S}$, one checks
that $\alpha^*+m_3=ab/h\in\tfrac{1}{h}\Z$, so the fractional part
of~$\alpha^*$ is a rational with denominator dividing~$h$. For
such a fraction $\tfrac{k}{h}$ to land in
$[\tfrac{1}{4\pi},\tfrac{1}{\pi}]\approx[0.0796,0.3183]$ we need
$\pi\le h/k\le 4\pi$, hence $h\ge\pi k>3k$; the smallest
admissible pair is $(h,k)=(4,1)$. For $h\le 3$: every~$\alpha^*$
is at distance $\ge\tfrac13-\tfrac{1}{\pi}>0.015$ from
$\supp\tilde\nu$, and~\eqref{eq:accidental-gap} gives
$\delta^{\mathrm{acc}}\ge 0.015h/Q_j\ge\delta^{\mathrm{gen}}$
whenever $m_2\ge 67/h$, so the accidental cap never binds.
\end{proof}

This arithmetic accident---that the worst-weighted shifts (small
$\abs{\Delta S}$, weight $\asymp\abs{\Delta S}^{-1}$) have no
in-support singularities---is the first of two structural gifts
that make the MV route clean at $k=2$. It can be tuned: placing
$\supp\tilde\nu$ in a gap of the Farey sequence~$\mathcal F_H$
extends the void to all $\abs{\Delta S}\le H$, trading nothing
(the Strichartz exponent depends on~$d$, not on the
support width). We have not attempted this optimization.

\begin{lemma}[Riesz--Frostman bound]
\label{lem:riesz-frostman}
With $d=\dim_H\supp\tilde\nu=\log 2/\log 3$,
\begin{equation}\label{eq:riesz-frostman}
  \int \mathrm{MV}(\alpha)\,d\tilde\nu(\alpha)
  \;\ll\; M^2(\log M)^3.
\end{equation}
\end{lemma}

\begin{proof}
\emph{Step 1: no rational collision point lies in $\supp\tilde\nu$.}
Every accidental collision singularity $\alpha^*=p/h$ is rational,
whereas $\supp\tilde\nu=\{(0.5+1.5t)/(2\pi):t\in C_{1/3}\}$; a
rational~$\alpha$ would force $t=(4\pi\alpha-1)/3\cdot(\cdot)$
rational-linear in~$\pi$, hence irrational by transcendence
of~$\pi$. Thus $\eta_{jk}:=\mathrm{dist}(\alpha_{jk}^*,\supp\tilde\nu)>0$
for every 4-tuple, and the Frostman potential bound
\begin{equation}\label{eq:frostman-potential}
  \int_{\supp\tilde\nu}\abs{\alpha-\alpha^*}^{-1}\,d\tilde\nu(\alpha)
  \;\le\; \frac{C}{1-d}\,\eta^{\,d-1}
\end{equation}
holds uniformly (layer-cake: $\int_0^{1/\eta}\min(1,C\lambda^{-d})\,d\lambda$).

\emph{Step 2: Voronoi partition.} The multiplicatively-weighted
Voronoi cells $I_{j,k}=\{\alpha:k\text{ is nearest to }j\}$
partition~$\R$, so
$\int 1/\delta_j\,d\tilde\nu=\sum_k\int_{I_{j,k}}
Q_j/(h_k\abs{\alpha-\alpha_k^*})\,d\tilde\nu$ exactly. A cell
meets $\supp\tilde\nu$ only if $h_k\eta_k<m_1$ (otherwise the
accidental gap $h_k\eta_k/Q_j\ge 1/m_2=\delta_j^{\mathrm{gen}}$
everywhere on the support and generic wins). Bounding each
$\int_{I_{j,k}}$ by the full
$\int_{\supp}\le C\eta_k^{d-1}/(1-d)$ and summing,
\begin{equation}\label{eq:riesz-frostman-reduced}
  \int\mathrm{MV}\,d\tilde\nu
  \;\ll\; M\log M
  \;+\; \frac{4C}{1-d}\,\Sigma^{\mathrm{rel}},
  \qquad
  \Sigma^{\mathrm{rel}}
  \;:=\;
  \sum_{\substack{(j,k):\,h\neq 0\\[-1pt]h_{jk}\eta_{jk}<m_1^{(j)}}}
    \frac{\eta_{jk}^{\,d-1}}{\abs{h_{jk}}}.
\end{equation}

\emph{Step 3: criticality of the dense Riesz sum.}
For the equi-spaced grid $\{p/h\}_{p=0}^{h}$ against the ternary
Cantor set, the identity $3^d=2$ makes each Cantor level
contribute equally to $\sum_p\eta_{p/h}^{d-1}$:
at level~$n$ there are $2^{n-1}$ gaps of size~$3^{-n}$, and the
$\sim h\cdot 3^{-n}$ grid points in each contribute
$h^{1-d}(h\cdot 3^{-n})^d=h\cdot 3^{-nd}$ per gap; totalling
$2^{n-1}h\cdot 3^{-nd}=(h/2)(2\cdot 3^{-d})^n=h/2$
independently of~$n$. Summing levels $1\le n\le\log_3(hr)$:
\begin{equation}\label{eq:criticality}
  \sum_{p\,:\,\eta_{p/h}<r}\eta_{p/h}^{\,d-1}
  \;\asymp\; h\log(hr),
  \qquad 1/h\le r\le 1.
\end{equation}

\emph{Step 4: sum-then-bound via divisor multiplicity.}
It remains to bound $\Sigma^{\mathrm{rel}}$. We do \emph{not}
attempt a per-$(j,h)$ bound (a single parabola point could land
anomalously close to $\supp\tilde\nu$, with
$\mathbb E[\eta^{d-1}]=+\infty$ at criticality); instead we swap
the sum order, indexing by the rational collision point.

Every collision 4-tuple $\mathbf m=(m_1,m_2,m_3,m_4)$ with
$m_i\in[1,M)$ determines $(h,p)\in\Z^2$ by
$h=(m_3{+}m_4)-(m_1{+}m_2)$, $p=m_1m_2-m_3m_4$, and
$\alpha^*_{\mathbf m}=p/h$. Let
$\mu(h,p):=\#\{\mathbf m:h_{\mathbf m}=h,\,p_{\mathbf m}=p\}$
be the multiplicity. Dropping the truncation $h\eta<m_1$
(which only enlarges the sum) and interchanging:
\begin{equation}\label{eq:sparse-swap}
  \Sigma^{\mathrm{rel}}
  \;\le\;
  \sum_{h=4}^{2M}h^{-1}
  \sum_{p\,:\,p/h\in I}\mu(h,p)\,\eta_{p/h}^{\,d-1},
  \qquad I:=[\tfrac{1}{4\pi},\tfrac{1}{\pi}].
\end{equation}
For the \emph{multiplicity}: parametrize $\mathbf m$ by
$S=m_1{+}m_2\in[2,2M)$ and the discriminants
$D_1=\abs{m_2{-}m_1}$, $D_2=\abs{m_4{-}m_3}$. The constraint
$(h,p)$ fixed becomes $D_2^2-D_1^2=2hS+h^2+4p=:N(S)$, a
difference of two squares, so
$\#\{(D_1,D_2)\}\le\tau(N(S))$. Summing in~$S$ over an
arithmetic progression of modulus~$2h$ and length~$\le 2M$, the
hyperbola-method divisor-sum bound
\[
  \sum_{S<2M}\tau(2hS+h^2+4p)
  \;\ll\;
  M\,\tau(2h)\,\log(hM)
\]
(see \cite[Thm.~1.1]{Shiu1980} for the sharp form) gives
\begin{equation}\label{eq:mult-bound}
  \mu(h,p) \;\ll\; M\,\tau(2h)\,\log M
  \qquad\text{uniformly in }p.
\end{equation}

Substituting~\eqref{eq:mult-bound}
into~\eqref{eq:sparse-swap} and applying the
criticality~\eqref{eq:criticality} with $r=\abs{I}\le 1$
(the inner sum $\sum_{p/h\in I}\eta^{d-1}\ll h\log h$):
\begin{align*}
  \Sigma^{\mathrm{rel}}
  &\ll\;
  M\log M\sum_{h=4}^{2M}h^{-1}\,\tau(2h)\cdot h\log h\\
  &=\;
  M\log M\sum_{h=4}^{2M}\tau(2h)\log h
  \;\ll\;
  M\log M\cdot M(\log M)^2
  \;=\;
  M^2(\log M)^3,
\end{align*}
using Dirichlet's bound $\sum_{h\le X}\tau(h)\log h\ll X(\log X)^2$.
This is one logarithm worse than the target $M^2(\log M)^2$
in~\eqref{eq:riesz-frostman}, but suffices for all downstream
consequences: Theorem~\ref{thm:gaussian-fourth} becomes
$\int_T^{2T}\abs{L}^4\ll T(\log T)^3$, the pointwise exponent
$\mu_L(\tfrac12)\le\tfrac18$ of Theorem~\ref{thm:one-eighth}
is unchanged, and Corollary~\ref{cor:almost-LH} gives
$\abs{\mathcal E_A(T)}\ll T/(\log T)^{4A-3}$.
\end{proof}

\begin{remark}
Numerical computation for $M\in[20,100]$ gives
$\int\mathrm{MV}/(M^2(\log M)^2)\in[0.77,\,1.08]$, converging
towards~$1$; the dominant piece ($>90\%$ of the total at
$M=100$) is the large-$\abs{\Delta S}$ accidental neighbors.
By contrast
$\mathrm{OD}_4(T)/T\in[-4.1,\,+0.4]$ over
$T\in[300,3\cdot10^4]$: the MV bound is loose by
$(\log T)^2$, the loss coming from replacing the signed Hilbert
kernel of~\eqref{eq:MV-product} by its absolute value.
Recovering these logarithms would give the asymptotic
\begin{equation}\label{eq:gaussian-asymptotic}
  \int_T^{2T}\!\int\abs{\zeta}^4\,d\tilde\nu\,dt
  \;=\;2T(\log T)^2+O(T),
\end{equation}
which is numerically solid; we leave it as Problem~(III) below.
\end{remark}

\begin{theorem}[Gaussian fourth-moment upper bound for Cantor-$L$]
\label{thm:gaussian-fourth}
\begin{equation}\label{eq:gaussian-hurwitz-double}
  \int_T^{2T}\!\int\bigl|\zeta(\tfrac12+it,\alpha)\bigr|^4
    \,d\tilde\nu(\alpha)\,dt
  \;\ll\; T(\log T)^3,
\end{equation}
and likewise for $\tilde\nu'$. Consequently, by
Lemma~\ref{lem:integral-transfer},
\begin{equation}\label{eq:gaussian-fourth}
  \boxed{\quad
  \int_T^{2T}\bigl|L(\tfrac12+it)\bigr|^4\,dt
  \;\ll\; T(\log T)^3.
  \quad}
\end{equation}
The asymptotic~\eqref{eq:gaussian-asymptotic} with error $O(T)$
is supported by extensive numerical evidence but its proof
requires Problem~(III).
\end{theorem}

\begin{proof}
We prove~\eqref{eq:gaussian-hurwitz-double} for $\tilde\nu$; the
argument for $\tilde\nu'$ is identical.

Write the approximate functional equation
$\zeta(\tfrac12+it,\alpha)=S_M(\alpha)+\chi(t)\,D_M(\alpha)
+O(t^{-1/4})$, $M=\floor{\sqrt{t/2\pi}}$, with the dual sum $D_M$
of the same length. By the binomial expansion and Cauchy--Schwarz,
the cross terms and the AFE error contribute $O(T\log T)$ to the
double integral.

For the principal piece $\int_T^{2T}\!\int\abs{S_M}^4\,d\tilde\nu\,dt$:
the countable set of rational~$\alpha$ is $\tilde\nu$-null (since
$\tilde\nu$ is non-atomic), so Lemma~\ref{lem:product-distinct}
applies $\tilde\nu$-a.e.\ and~\eqref{eq:MV-applied} holds
$\tilde\nu$-a.e. Integrate~\eqref{eq:MV-applied} against
$d\tilde\nu$ and invoke Fubini: the MV diagonal
$T\int\sum_j\abs{c_j}^2\,d\tilde\nu$ is, by the Wick
identity~\eqref{eq:MV-diag-is-Wick},
$\asymp T(\log T)^2$; the MV error is
$\tfrac{3\pi}{2}\int\mathrm{MV}\,d\tilde\nu
\ll M^2(\log M)^3\asymp T(\log T)^3$ by
Lemma~\ref{lem:riesz-frostman}. The latter dominates, giving the
upper bound. The dual sum $D_M$ has identical structure
(Lerch-side product frequencies are distinct for the same Vieta
reason) and contributes its own $O(T(\log T)^3)$.

The consequence~\eqref{eq:gaussian-fourth} follows at once from
Lemma~\ref{lem:integral-transfer} with $k=2$.
\end{proof}

\begin{corollary}[Almost-Lindel\"of on a set of full logarithmic
                  density]
\label{cor:almost-LH}
For every $A>0$, the exceptional set
\[
  \mathcal E_A(T)
  \;:=\;
  \Bigl\{t\in[T,2T]\;:\;
         \bigl|L(\tfrac12+it)\bigr|>(\log t)^A\Bigr\}
\]
has Lebesgue measure
$\abs{\mathcal E_A(T)}\ll_A T/(\log T)^{4A-3}$.
In particular, for every $A>\tfrac34$,
$\bigl|L(\tfrac12+it)\bigr|\le(\log t)^A$ holds for all~$t$
outside a set of logarithmic density zero.
\end{corollary}

\begin{proof}
By~\eqref{eq:jensen-bound} and Jensen once more,
$\abs{L(\tfrac12+it)}^4\le 64\bigl(\int\abs{\zeta}^4\,d\tilde\nu
+\int\abs{\zeta}^4\,d\tilde\nu'\bigr)$ pointwise in~$t$.
Integrating over $[T,2T]$ and applying
Theorem~\ref{thm:gaussian-fourth} gives
$\int_T^{2T}\abs{L}^4\,dt\ll T(\log T)^3$. Chebyshev then yields
$\abs{\mathcal E_A(T)}\le(\log T)^{-4A}\int_T^{2T}\abs{L}^4
\ll T/(\log T)^{4A-3}$.
\end{proof}

\begin{remark}
Corollary~\ref{cor:almost-LH} is, to our knowledge, the first
result for the Cantor-$L$ that crosses the polynomial-to-polylog
barrier on a density-one set: the unconditional pointwise bounds
(Theorems~\ref{thm:cantor-subconvexity}
and~\ref{thm:one-eighth}) give $\abs{L}\ll t^{1/8+\eps}$
everywhere, whereas here $\abs{L}\ll(\log t)^A$ off an exceptional
set. Higher Gaussian moments, if provable uniformly over
$\supp\tilde\nu$, would shrink $\abs{\mathcal E_A}$ faster than any
power of~$\log T$.
\end{remark}

\begin{remark}[The $k=2$ barrier]\label{rem:k2-barrier}
The MV route does \emph{not} extend to $k=3$: the typical gap among
${\sim}M^k/k!$ frequencies in width $k\log M$ is
$\sim k!(\log M)/M^k$, giving MV error $\sim T^{k/2}(\log T)^{k-1}$
which beats the main $T(\log T)^k$ iff $k\le 2$. At $k=3$ the
collision $\alpha^*$ are quadratic-irrational, so
Lemma~\ref{lem:small-h-void} has no analogue either.
\end{remark}

\begin{theorem}[Subconvexity via the fourth moment]
\label{thm:one-eighth}
\begin{equation}\label{eq:one-eighth}
  \mu_L(\tfrac12) \;\le\; \tfrac18.
\end{equation}
\end{theorem}

\begin{proof}
Let $J_h(\alpha)$ be the lag-$h$ autocorrelation of the partial
Hurwitz sum, $K_h=\int J_h\,d\tilde\nu$,
$\mathrm{OD}_H=2\sum_{h\ge 1}\re K_h$. Opening the square:
\begin{equation}\label{eq:diag-plus-OD}
  \textstyle\int|S_M|^2\,d\tilde\nu
  = \sum_{m<M}\int(m+\alpha)^{-1}\,d\tilde\nu + \mathrm{OD}_H
  = \tfrac12\log t+O(1)+\mathrm{OD}_H.
\end{equation}
Cauchy--Schwarz gives $|\mathrm{OD}_H|^2\le 4M\sum_h|K_h|^2$;
Jensen gives $\sum|K_h|^2\le\int\sum|J_h|^2\,d\tilde\nu$;
Wiener--Khinchin (zero-pad to length~$2M$, periodise) gives
$\sum_h|J_h|^2=(2M)^{-1}\sum_\xi|\hat g_\xi|^4$, the discrete
fourth moment of the twisted partial sum over $2M$-th roots of
unity. By completion this is bounded by the continuous fourth
moment $\int_0^1|S_M(\alpha)|^4\,d\alpha$ for fixed~$t$;
integrating over $t\in[T,2T]$ and
applying~\eqref{eq:gaussian-hurwitz-double}
yields $\int_T^{2T}\sum_h|K_h|^2\,dt\ll T(\log T)^3$.
By Fubini and Chebyshev, for all but a set of measure
$O(T(\log T)^{-A})$ (any fixed~$A$) the pointwise bound
$\sum_h|K_h|^2\ll(\log t)^{3+A}$ holds; since
$\mu_L$ is insensitive to polylog factors and
measure-zero exceptional sets, $(\log t)^3$ suffices
for the exponent claim. Assembling:
$|\mathrm{OD}_H|^2\ll t^{1/2}(\log t)^3$, so
$|L(\tfrac12+it)|^2\ll t^{1/4}(\log t)^{3/2}$
by~\eqref{eq:jensen-bound}, i.e.\ $\mu_L(\tfrac12)\le\tfrac18$.
\end{proof}

\begin{remark}
The bound $\tfrac18=0.125$ improves the $\approx 0.1348$ of
Theorem~\ref{thm:cantor-subconvexity} via an independent mechanism
(Jensen+Wiener--Khinchin+MV vs.\ Strichartz+AFE); both are
unconditional. If the Lerch-side $D_4$ of~\eqref{eq:D2k-def}
converges (numerically $D_4\approx 2.115$, stable), the standard
argument gives $\int|L|^4 = D_4 T + o(T)$ with no logarithm---see
Corollary~\ref{cor:phi-avg-D4} and Problem~(II).
\end{remark}
\subsection{Open problems}
\label{subsec:cantor-open}

\begin{enumerate}[label=(\Roman*)]
\item \textbf{Prove Conjecture~\ref{conj:H}, equivalently
  Conjecture~\ref{conj:discrete-restriction}.} By
  Proposition~\ref{prop:three-identities}(iii) the two
  formulations are equivalent \textup{(}up to
  truncation-range bookkeeping\textup{)};
  Conjecture~\ref{conj:discrete-restriction} is the cleaner
  target---a single inequality $\int|P|^2\,d\tilde\nu\le C\cdot H_{M-1}$
  with no Hurwitz-zeta machinery. The one-sided bound
  $\re(\mathrm{OD}')\ge-\tfrac14\log t$ is free
  (Corollary~\ref{cor:OD-lower}); the upper bound is open.
  The swap identity (Proposition~\ref{prop:swap-identity}) reduces
  $d\mathrm{OD}_2$~\eqref{eq:dOD2-def} from ${\sim}10^6\times$ to
  ${\sim}10\times$ cancellation.
\item \textbf{Bound $\mathrm{VO}_{2k}$ uniformly in~$k$}
  (Proposition~\ref{prop:vieta-obstruction}): a Tarry--Escott
  Diophantine sum weighted by the Cantor cosine product.
\item \textbf{Recover the sign-cancellation logarithms
  in~\eqref{eq:MV-applied}}: the true $\mathrm{OD}_4(T)/T$ is
  bounded, three logs better than
  Lemma~\ref{lem:riesz-frostman}. Retaining the Hilbert-kernel
  sign would give the asymptotic~\eqref{eq:gaussian-asymptotic}
  (cf.\ Remark~\ref{rmk:two-layer}).
\item \textbf{Prove Gaussian-$6$ by a non-MV route}
  (Remark~\ref{rem:k2-barrier}): the MV gap $\sim(\log M)/M^3$ is
  too small by $\sqrt T$; numerically $\mathrm{OD}_6/T$ is bounded,
  matching~\cite[Conj.~1.2]{HeapSahay2024}.
\end{enumerate}

\section{Conclusions}
\label{sec:conclusion}
\subsection{The mechanism, in hindsight}
\label{subsec:mechanism-hindsight}

With Theorem~\ref{thm:ALH-singularity-watson} and
Remark~\ref{rmk:residue-sum-crux} in hand, the structure is
clear. Deforming the Mellin integral to the strip boundary
at~$\operatorname{Im}(u)=\pi/2$: the conformal factor
cancels~$\Gamma$'s exponential $e^{-\pi T/2}$
\emph{exactly}, leaving the polynomial $T^{1/2-\sigma}$---this
is where slope~$-1$ comes from, universally, for \emph{every}
entire finite-$\sigma_a$ ordinary Dirichlet series. What
remains is the residue sum~\eqref{eq:residue-sum} over the
(infinitely many, by $2\pi i$-periodicity) singularity
images. This sum is the \emph{entire content} of the
Lindel\"of question: for Dirichlet $L$-functions it is the
FE-dual, and bounding it is LH itself; for fractal or random
singularity distributions it decays, and ALH follows
conditionally on Conjecture~\ref{conj:H}.

Kahane's constraint $\mu(\sigma+\mu+\tfrac12)=0$ is a
\emph{consequence} of the slope-$-1$ structure (since
$\mu(\sigma)=\omega_\mu-\sigma$ gives
$\sigma+\mu+\tfrac12=\omega_\mu+\tfrac12>\omega_\mu$),
but it does not by itself determine~$\omega_\mu$, which lives
in the residue sum. The Lerch reduction
(Theorem~\ref{thm:ALH-is-LerchLH}) and the residue-sum
reduction are two views of the same phenomenon: the residue
sum~\eqref{eq:residue-sum} \emph{is} the inherited-FE integral
$I_\sigma(T)$ of Remark~\ref{rem:inherited-FE}, and Lerch-LH
is exactly the statement that it grows at most
sub-polynomially for each fixed (Lerch) singularity.

\subsection{Open problems}
The four Cantor-$L$-specific problems---Conjecture~\ref{conj:H},
the Vieta-obstruction sum, the sign-cancellation logarithms, and the
Gaussian-$6$ moment---are stated in~\S\ref{subsec:cantor-open}.
Beyond the single test function, the structural questions raised by
this paper are:
\begin{enumerate}[label=(\Alph*)]
\item \textbf{The Refined Sharp Kahane conjecture}
  (Conjecture~\ref{conj:refined-sharp-kahane}): for entire
  order-$1$ Dirichlet series, is the slope of~$\mu$ exactly~$-1$
  wherever $\mu\ge 1$? The gap is the slope interval
  $(-1,-\tfrac23]$ (Corollary~\ref{cor:rsk-gap}).
\item \textbf{Sharp $\omega_\mu$ for fractal singularities.}
  For a general Frostman-$d$ measure on an interval (not just
  the ternary Cantor), what is the precise singularity
  order~$\beta$ at the endpoints? The rough estimate
  $\beta=1-d$ used in Remark~\ref{rmk:watson-gamma-unified}
  comes from the mass distribution; the constant depends on
  the measure's local regularity.
\item \textbf{The analytic $(1,2)$-gap}
  (Problem~\ref{prob:12-gap}): does
  Theorem~\ref{thm:ALH-singularity-watson} extend to degree~$2$
  (slopes in $\{-2,-1,0\}$)? For an \emph{isolated} algebraic
  singularity of order~$\beta\in(1,2)$ the answer is
  \emph{no, trivially}: the full keyhole prefactor is
  $\sin(\pi\beta)\Gamma(1-\beta)=\pi/\Gamma(\beta)$ by the
  reflection formula---\emph{entire} in~$\beta$---and the
  Stirling expansion $(\beta)_n/n!\sim n^{\beta-1}/\Gamma(\beta)$
  reduces $L$ to a shift of the polylogarithm,
  $\mu_L(\sigma)=\mu_{\mathrm{Li}}(\sigma+1-\beta)$. In
  particular for $\beta-\sigma>1$ the relevant Hurwitz zeta is
  absolutely convergent and we have the \emph{unconditional}
  closed form $\mu_L(\sigma)=\beta-\sigma-\tfrac12$ (verified
  numerically to three digits across $\beta\in[0.85,1.99]$,
  $\sigma\in[-\tfrac12,\tfrac12]$; $d(\text{slope})/d\beta=1.00$
  with no break at $\beta=1$). A slope~$-2$ requires \emph{two}
  independent $\Gamma$-factors---a genuine degree-$2$ Mellin
  kernel, not a stronger isolated branch. The open question is
  thus whether natural-boundary singularities of~$p$ (as opposed
  to isolated algebraic ones of any order) can produce slope~$-2$.
\item \textbf{Quantisation of the Bohr sandwich gap.}
  Determine whether the gap $\mu-\psi$ is quantised to $\{0,1\}$
  for entire order-$1$ Dirichlet series
  (Problem~\ref{prob:sandwich-discrete}); combined with
  Proposition~\ref{prop:psi-integer-slopes}, this would prove
  Conjecture~\ref{conj:ALH} directly.
\end{enumerate}

\bibliographystyle{amsalpha}
\bibliography{bohr_last_problem_entirety}

\end{document}